\documentclass[11pt,a4paper]{article}

\usepackage[utf8]{inputenc}
\usepackage[T1]{fontenc}
\usepackage[english]{babel}

\usepackage{xcolor}
\usepackage{pgf,tikz}
\usetikzlibrary{arrows}
\usetikzlibrary{decorations.markings,decorations.pathmorphing,patterns,positioning}

\usepackage[format=hang,font=small]{caption}
\usepackage{booktabs,bm,multirow,subfig,comment,array}%
\usepackage{enumitem}

\usepackage{geometry}
\usepackage{amsmath,amsfonts,amssymb,amsthm}%
\usepackage{dsfont,multicol}

\usepackage
{graphicx}

\graphicspath{{Images/}}

\usepackage{fancyhdr}
\renewcommand{\phi}{\varphi}
\renewcommand{\theta}{\vartheta}
\renewcommand{\epsilon}{\varepsilon}

\newcommand{\field}[1]{\mathbb{#1}} 
\newcommand{\R}{\field{R}}

\newcommand{\A}{\field{A}}
\newcommand{\I}{\field{I}}

\newcommand{\CC}{\mathcal{C}}
\newcommand{\EE}{\mathcal{E}}
\newcommand{\RR}{\mathcal{R}}

\newcommand{\NN}{\mathcal{N}}

\newcommand{\Sph}{\mathcal{S}}

\newcommand{\vm}{v_\mathrm{m}}
\newcommand{\Clip}{\CC^\mathrm{Lip}}
\newcommand{\Ash}{\A_\mathrm{sh}}
\newcommand{\Esh}{\EE_\mathrm{sh}}
\newcommand{\Dsh}{\RR_\mathrm{sh}}
\newcommand{\lsh}{\ell_\mathrm{sh}}
\newcommand{\Psh}{\Psi_\mathrm{sh}}
\newcommand{\C}{{C}}
\newcommand{\Csh}{C_\mathrm{sh}}
\newcommand{\hCsh}{\hat C_\mathrm{sh}}
\newcommand{\mud}{\mu^\mathrm{d}}
\newcommand{\mum}{\mu_\mathrm{min}}

\newcommand{\Lmax}{L^\mathrm{max}}
\newcommand{\ten}{\Sigma}
\newcommand{\cirlet}[1]{\textcircled{\textsf{#1}}}
\DeclareMathOperator{\dist}{dist}

\DeclareMathOperator{\meas}{meas}

\DeclareMathOperator*{\argmin}{arg\,min}

\providecommand{\abs}[1]{\left\lvert#1\right\rvert}                             
\providecommand{\norm}[1]{\left\lVert#1\right\rVert}    
\newcommand{\scal}[2]{\left\langle #1,#2\right\rangle}

\newcommand*{\diff}{\mathop{}\!\mathrm{d}}
\newcommand{\dd}{\diff}
 
\newtheorem{theorem}{Theorem}
\newtheorem{lemma}[theorem]{Lemma}

\theoremstyle{definition}
\newtheorem{defin}[theorem]{Definition}
\newtheorem{remark}[theorem]{Remark}
\newtheorem{example}[theorem]{Example}

\numberwithin{equation}{section}
\numberwithin{theorem}{section}

\usepackage[babel]{csquotes}
\usepackage[style=numeric,giveninits=true, backend=biber, doi=false,isbn=false,url=false,maxnames=4,minnames=2,sorting=nyt]{biblatex}
\renewbibmacro{in:}{%
	\ifentrytype{article}{}{%
		\printtext{\bibstring{in}\intitlepunct}}}

\bibliography{biblio_RIScrawler}

\definecolor{uuuuuu}{rgb}{0.25,0.25,0.25}

\title{Rate-independent soft crawlers}

\author{\textsc{Paolo Gidoni}}
\date{\small{\textit{CMAF-CIO - Centro de Matemática, Aplicações Fundamentais e Investigação Operacional, Universidade de Lisboa,  Portugal.\footnote{e-mail address:\,\texttt{pgidoni@fc.ul.pt} -- address: CMAF-CIO, Faculdade de Ciências da Universidade de Lisboa, Campo Grande, Edificio C6, 1749-016 Lisboa, Portugal}}
} 
}

\begin{document}

\maketitle

\begin{abstract}
	This paper applies the theory of rate-independent systems to model the locomotion of bio-mimetic soft crawlers. We prove the well-posedness of the approach and illustrate how the various strategies adopted by crawlers to achieve locomotion, such as friction anisotropy, complex shape changes and control on the friction coefficients, can be effectively described in terms of stasis domains.

	Compared to other rate-independent systems, locomotion 
models do not present any Dirichlet boundary condition, so that all rigid translations are admissible displacements, resulting in a non-coercivity of the energy term. We prove that existence and uniqueness of solution are guaranteed under suitable assumptions on the dissipation potential. Such results are then extended to the case of time-dependent dissipation. \smallskip
\end{abstract}


\section{Introduction}

Rate-independence is the property of input-output systems whose behaviour is preserved by time-reparametrizations, meaning that a time-rescaling in the input would make the output evolve along the same path of the original system, but according to the same time-rescaling of the input (cf.~\cite{MieRou15}).
Rate-independent systems appear in many phenomena. We mention, for instance, the quasi-static evolution of mechanical systems in presence of dissipative forces such as Coulomb dry friction, plasticity or damage; the motility of swimmers at low Reynolds number; the behaviour of ecological system subject to a slow change in an environmental parameter. 

The expression \emph{rate-independent systems} (RIS) has become widespread referring to a wide class of those phenomena, characterized by the following variational formulation:
\begin{equation}
0\in \partial \RR(\dot x(t))+D_x\EE(t,x)
\label{eq:RIS}
\end{equation}
Physically, this equation is a force balance: configurational forces (e.g.~tension in an elastic body) are described as the spatial gradient of the internal energy $\EE$, while frictional forces are obtained as the subdifferential of a dissipation potential $\RR$, assumed to be positively homogeneous of degree one, in order to guarantee rate-independence.

The variational structure of the problem has favoured the development of an advanced and extensive mathematical theory of rate-independent systems, assisting and inspired by applications in physics of solids and continuum mechanics, modelling phenomena such as elastoplasticity, fracture, damage, phase transitions.

Along with these, systems with dry friction has immediately emerged as one of the most fitting applications of the theory, already from  the first major achievements in the 70s, with the introduction of Moreau's Sweeping processes \cite{Moreau77}.
As today, dry friction still contributes to the development of the theory. We mention the recent results adopting multiscale approaches to study the effective friction in the case of fast-oscillating bodies \cite{HeiMie16} and of hairy surfaces \cite{GidDeS17}; or the illustration  of the properties of non-convex rate independent systems using using simple, representative  dry friction toy models~\cite{Alessi16}.

\medbreak

This paper proposes to apply the theory to another family of friction-based phenomena: \emph{crawling locomotion}.
Our aim is to provide a basic theoretical introduction to the motility of soft crawlers as rate independent system, discussing the main issues and illustrating, with elementary examples, the versatility of this approach. 
In this, we carry on the first step taken in \cite{GidDeS16}, where the motility of a toy model of crawler with three contact points was thoroughly analysed. 
The formulation of crawling as rate-independent system has the advantage to incorporate naturally two demanding, but often realistic, features of crawlers: elasticity of the body, with the consequence that shape is not controlled \emph{a priori}, and a discontinuous friction force. The diffusion of \emph{soft robotics} as paradigm in the design of robotic locomotors makes this class of models increasingly relevant. 

Moreover, the theory of rate-independent systems allows a comprehensive approach to such a  rich family of models. As we show in this paper, we can gather within the same theoretical structure both discrete and continuous models, on the line and on the plane, and including features as friction anisotropy and friction manipulation. We believe that our top-down approach, complementary to an extended and detailed literature of case-specific studies, could support, with a different perspective, the understanding and development of such locomotors.

Compared to classical applications of rate-independent systems, where the net position of the body is usually determined by some Dirichlet boundary conditions, in locomotion it becomes not only free, but the most relevant component to study in the evolution of the system. Indeed, we have no Dirichlet conditions and the actuation of the system is provided only by an internal active distortion, instead of having an external load. This also means that the energy $\EE$ presents an invariance for rigid translation on the space of admissible configurations. A major consequence of such lack of coercivity is that the uniqueness of solution for the system is not guaranteed any more. 

The main mathematical contribution of this paper is to provide, in such situations, sufficient conditions for existence and uniqueness of the solution, assuring the wellposedness of our models. This is done for quadratic energies in Section \ref{sec:indep_theory}, with Theorem \ref{th:indep}. Later, in Section~\ref{sec:dep_theory},  we extend our approach to the case of a time-dependent dissipation potential $\RR=\RR(t,\dot x)$, with Theorem \ref{th:dep}.

Such results are applied to the modelling of crawling locomotion in Sections \ref{sec:indep_examples} and  \ref{sec:dep_examples}. We consider both discrete and continuous models, verify that our assumptions are satisfied by actual locomotors, and briefly comment the nature of the critical cases in which multiplicity of solution may occur.
We then use our formulation of the problem to represent the motility properties of the crawler through its \emph{stasis domain}. This description, introduced in \cite{GidDeS16}, allows to analyse the locomotion capabilities without any assumption on the shape-change strategy adopted. The main models with time-independent are discussed in Section \ref{sec:indep_examples}. 
In Section \ref{sec:dep_examples},  instead, we focus on three variations of \emph{two-anchor crawling}, to emphasize the effectiveness of a control on the friction coefficients. We remark that our approach leaves a complete control on the friction coefficients, which can vary in a continuous range and independently of the actuation on shape, a condition often more realistic than a discrete switching between anchoring and unanchoring. Our examples show how the description through stasis domains encompasses the three main main mechanisms employed in crawling locomotion: friction anisotropy, complex shape changes and friction manipulation.

\medbreak
The study of the locomotion of soft crawlers raises several additional questions, such as optimization and stability, that spread beyond the extent of this paper. At the same time, the framework of rate-independent systems provides a mature theory, with a rich and advanced toolkit. We believe that a larger employment of those techniques, such as the introduction of state-dependence in the dissipation potential, or the adaptation of results recently obtained for the optimal control of sweeping processes \cite{CHHM12,CHHM15,CHHM16,ColPal16}, may support the study of those issues. 
In this perspective, we hope that this paper will contribute to a better understanding of the main ingredients and structures of the locomotion of soft crawlers as rate-independent system, providing solid grounds for future investigations.

\paragraph{Soft crawlers: from Nature to robotics}
 From anellids, as earthworms \cite{ACD04,Quillin1999} and leeches \cite{SNPK86}, to gastropods \cite{Lai_2010}, Nature offers us a large collection of crawling strategies. Each species implements the fundamental mechanism of crawling locomotion in its own way: not only we can observe large differences in the periodic shape change exploited to move \cite{Kuroda2014}, but we also find several elements, like mucus or setae, used to modify the frictional interaction with the environment.

The ability of such locomotor to obtain, with a simple and contained mechanism, a very good motility, working also in irregular or narrow environments, has drawn the attention of the robotics community, leading in recent years to the development several prototypes of  bio-mimetic robotic crawlers. In this trend we can identify two main design principles: the use of soft, elastic materials, and the minimization of the control parameters.

In contrast to robots built from rigid materials, animals are usually soft bodied, and exploit the elastic deformation of their bodies to help achieve a complex and compliant behaviour during locomotion or manipulation.
The paradigm of \emph{soft robotics} proposes to adopt the same principle by building machines from soft materials \cite{CalPicLas17,Kim,trivedi_soft_2008}. In this way, soft robots are endowed  with new capabilities in terms of dexterity or adaptability, supporting motility in an unpredictable and
complex environment. Different materials and design principles has been adopted in the implementation of the soft robotics paradigm. In the case of robotic crawlers we report the use of shape-memory alloys (SMA) coils acting on a silicon- \cite{Men06}, mesh- \cite{Manwell14,Seok2013} or origami-structure \cite{Onal2013}, soft actuators made of dielectric elastomers \cite{Jung2007,Xu2017},   pneumatic actuators \cite{FeiShen2013,Shepherd2011,Tolley14} or motor-tendon actuators \cite{Ume16b,Ume16c} combined with an elastic structure.
The use of robotic soft  crawlers have been proposed for medical intervention, in particular for colonscopy \cite{Bernth17,endotics1,Kassim06}, to reduce the discomfort and trauma in the patient. Moreover, their properties of compliance and resistance to damage \cite{Seok2013,Tolley14} encourage  future usage in reconnaissance missions involving rough terrain and hostile environment.

The \emph{morphological computation} of a soft body, interacting and adapting to the environment,  facilitates the pursuit of \emph{simplexity}, the ability of a system to achieve a complex behaviour with a reduced number of control parameters. This principle allows not only to fully exploit the advantages of a soft body, but also to facilitate the miniaturization of the device.
Our aim is therefore to understand  how a good movement capability can be achieved by the use of just a few basic mechanisms, such as contraction-expansion cycles of a segment. 
The inquiry for minimal mechanisms is, in fact, an active research theme in the study of motility, especially for miniaturized locomotors. We mention here, for instance, the case microswimmers, whose results present many analogies with our models. Indeed, the motility of simple swimmers made by few elements, like segments \cite{Zop13,Zop15} or 
joined spheres \cite{AlougesDeSimoneLefebvre2009,MonDeS17,NajGol2004}, has been extensively analysed.

\paragraph{Friction anisotropy, friction manipulation, and shape change strategy}
We  therefore enquire which are the minimal requisites for a crawler to move. 
Since motility means that a periodic input produces a non-zero net displacement after each period, some  source of asymmetry must be present in the system, in order to define the  direction of motion. Such element may be found in the frictional interaction with the environment, in the shape-change strategy, or be generated by a suitable combination of both.

Let us therefore start by considering friction as source of directionality. The general idea is quite intuitive: we expect stronger friction forces opposing backward motion, thus producing a firm grip that prevents the locomotor to slide in the undesired direction. This scheme can be implemented in two ways: directly, by producing an anisotropic force-velocity friction law, or indirectly, with an active control on the friction coefficients. 

The use of textured, hairy or scaly surfaces, in order to create a frictional anisotropy to improve their motility, is widespread in biological and bio-inspired locomotors \cite{FilGor13,Manoonpong2016}. In biological crawlers, we mention earthworms, which present bristle-like \emph{setae} opposing backward sliding \cite{Quillin1999}. In robotic crawlers, friction anisotropy is introduced using several mechanisms, such as elastic bristles \cite{Jung2007,NosDeS14}, hooks \cite{Men06} or directional clutches \cite{Vikas16}. 

As we show in Section \ref{sec:dep_examples}, an active control of the friction parameters, suitably synchronized with the shape change, can achieve the same effects of friction anisotropy, increasing the grip to oppose a thrust backwards and reducing the friction to slide forward. Having a control on friction has the additional advantage that the same mechanism can be exploited to support locomotion in both direction. Sometimes, such changes in the friction coefficients may be produced conjointly with the changes in shape of the crawler. In cylindrical crawlers with a hydrostatic \cite{Quillin1999} or a mesh-braided body \cite{Manwell14}, a longitudinal contraction is associated to a thickening of the contracting segment; when moving in a tube, this corresponds to an increase in friction. In \cite{Ume13,Vikas15,Vikas16} a friction manipulation mechanisms is used in a crawler with two contact points.  A capsule made of two materials, one hard and slippery while the other soft and sticky, is placed at each of the two extremities of the crawler. Changes in the shape of the crawler, controlled by two parameters, not only produce a contraction-expansion cycle, but also produce a change in the tilt angle at the extremities, determining whether the contact occurs in the hard or soft part of the capsule. Also anchoring mechanisms, employed by caterpillars \cite{GriTri14} and miriapods \cite{Kuroda2014}, can be considered as extreme tools of friction control.  Moreover, as shown in \cite{GidDeS17}, a control on the friction coefficients can be obtained also in the case of anisotropic friction produced by bristles, with a change in the rest configuration or in the stiffness of the bristle. This happens in earthworms, where longitudinal contractions also produce an extension of the setae \cite{Quillin1999}. Finally, also snakes can influence the friction coefficients by changing the attitude of their scales; an analogous mechanisms is used by the crawling robot presented in \cite{scalybot}.

A third way to introduce directionality in the system consists of a sufficiently complex shape change, described by at least two parameters (e.g.\ a circular orbit in the space of shape control parameters). Indeed, in case of anisotropic friction a suitable periodic shape-change may be able to move the crawler also in the direction of higher opposing friction, as shown in \cite{BSZZB17,GidNosDeS14,GidDeS16}. Nevertheless, also in this case anisotropic friction can facilitate motility in the chosen direction, as intuitively imaginable. This is well illustrate, for instance, by the qualitative changes in the motility of the crawler in Example \ref{example:3legs} with respect to the ratio of the directional friction coefficients. It is therefore not surprising that anisotropic friction can be found in natural and robotic crawlers, also when not strictly necessary. On the other hand, we point out that  directionality can be effectively produced by shape change alone only in presence of sufficient nonlinearity in the system. In the framework of our paper, such nonlinearity is provided by Coulomb dry friction. On the other hand, if we replace this assumption with linear (viscous) friction, more appealing by a mathematical point of view, this change would dramatically affect the behaviour of the system. As illustrated by the \lq\lq incompetent crawler\rq\rq\ in \cite{DeSTat12}, locomotors, similar to that presented in Example \ref{example:3legs},  lose their ability to move if we consider viscous friction. To move under viscous friction, a more complex structure is necessary, where sufficient nonlinearity is provided, for instance by the effects on friction of large deformations in continuous models of crawlers, or by non-local hydrodynamic interactions in a three-spheres swimmer \cite{NajGol2004}.

Finally, we mention a fourth way to produce anisotropy in the system, even in the case of isotropic, constant friction and a single shape-parameter. As shown in \cite{WagLau13}, this is possible with a shape change that is not time-reversible, an asymmetric body, and exploiting inertial effects. However, in this case locomotion occurs at high Froude numbers, meaning that inertial forces (and therefore velocities) are much larger than friction forces. This scenario is uncommon in crawlers, where we observe a stick-slip behaviour, instead of the permanent sliding required by such inertial strategies. Likely, this is due to the great energetic costs required by sustaining high velocities. For these reasons, the quasi-static approximation is customary in the analysis of crawling locomotion, and it  is assumed in our framework.


\section{Crawling with time-independent friction}
\label{sec:indep_theory}

\subsection{A general formulation of crawling} \label{sec:sub_indep}
We describe the body of the crawler, in the reference configuration, with a set $\Omega\subset \R^d$.  We will consider both discrete models, where $\Omega$ is a finite collection of points, and continuous models, where $\Omega$ is an interval.
Given a point $\xi\in \Omega$, we denote with $x(t;\xi)$ its displacement at time $t$, meaning that at time $t$ the point $\xi$ of the crawler will be positioned at $\xi+x(t;\xi)$. 
In this way we represent  the displacement of the body of crawler at any time $t$   by a vector $x=x(t;\cdot)\in X$, where $X$ is a suitable Hilbert space.  To distinguish between the  displacement produced by a net translation of the crawler and that associated to an internal shape change, we identify the space $X$ as the product of two subspaces $Z\times Y$, with $Y=\R^d$, and $Z$ a Hilbert space (possibly with a different scalar product). Each vector $x\in X$ can be identified with the couple $(z,y)$, where $y\in Y$ and $z\in Z$. This decomposition has to be interpreted as follows. The vector $y\in Y$ is associated to a generic description of the position of the crawler, e.g.\ the position of its head or of its centre. For most of the paper we will consider the case $d=1$, corresponding to a crawler moving on a line. We will discuss briefly the planar case $d=2$ in Section \ref{sec:planarcrawler}.  The vector $z\in Z$ instead describes the shape of the crawler. In other words,  the absolute displacement $x$ of the different parts of the body of the crawler can be described by the absolute position $y$ of a specific point, together with the relative positions $z$ of the other parts of the crawler. 
The rationale of this decomposition is the invariance of the configurational energy $\EE$ of the system with respect to translations of the crawler, meaning that $\EE$ depends only on the shape $z$ and on a time-dependent load exerted by the crawler itself.

We denote with $\sigma \colon X\to Z$ the map that associates to the displacement $x\in X$ of the crawler its corresponding shape deviation $z\in Z$ from the reference shape. Similarly, we define $\pi \colon X\to Y$, that will describe the net translation $y=\pi(x)\in Y$ of the crawler associated to a displacement $x$. Given a shape deviation $z\in Z$ and a net translation $y\in Y$, we denote the associated displacement of the crawler as $\chi(z,y)\in X$, so that $x=\chi(\sigma(x),\pi(x))$. We assume  that the maps $\sigma$, $\pi$ and $\chi$ are continuous linear operators. We remark that, however, the maps are not necessarily orthogonal and therefore the scalar product on $Z$ is not necessarily the one induced by $X$ through $\sigma$. 
As we will see, this notation is justified by finite dimensional models (cf.~Section \ref{sec:disc_indep}), where it is much more natural and convenient to adopt distinct coordinates on $X$ and on $Z\times Y$, not orthogonal with respect to each other.

The invariance for translations of the internal, or configurational, energy $\EE$ means that it can be expressed only in terms of the current shape, namely
\begin{equation}
\EE(t,x)=\Esh(t, \sigma (x))
\end{equation}
with $\Esh\colon[0,T]\times Z \to \R$. We assume that the shape-restricted energy $\Esh$ has the form
\begin{equation} \label{eq:ener_shape}
\Esh(t,z)= \scal{\Ash z}{z}-\scal{\lsh(t)}{z}
\end{equation}
where $\Ash \colon Z\to Z^*$ is a symmetric linear positive definite operator and $\lsh(t)\colon [0,T]\to Z^*$ is  continuously differentiable and with values in the subspace $Z$. Here, for a generic Hilbert space $\mathcal X$, we are denoting with $\mathcal X^*$ the dual space of $\mathcal X$ and with $\scal{\cdot}{\cdot}\colon \mathcal X^*\times \mathcal X\to \R$ the corresponding dual pairing. 

As a consequence, the the configurational energy, expressed in terms of the displacement $x$, takes the form
\begin{equation} \label{eq:energy}
\EE(t,x)= \scal{\A x}{x}-\scal{\ell(t)}{x}
\end{equation}
Here $\A\colon X \to X^*$ is a linear symmetric positive-semidefinite operator, defined by $\scal{\A x}{u}=\scal{\Ash \sigma(x)}{\sigma (u)}$ for every $x,u\in X$. As a consequence, $\ker \A$ corresponds  to pure  translations, namely
$x\in \ker \A$ if and only if $\sigma(x)=0$. Analogously, $\ell(t)\colon [0,T]\to X^*$ is defined as $\scal{\ell(t)}{x}=\scal{\lsh(t)}{\sigma(x)}$ for every $x\in X$
.

\medbreak
We assume a rate-independent dissipation, that in our case will be produced by Coulomb dry friction, so that the forces are described by a dissipation potential $\RR\colon X\to [0,+\infty]$, that we assume proper, convex, lower semi-continuous, and positively homogeneous of degree one. We recall that a function $f\colon X\to [0,+\infty]$ is \emph{proper} if it assumes a finite value in at least one point.

Thus the evolution of the system is described by the force balance
\begin{equation}
0\in \partial \RR(\dot x(t))+D_x\EE(t,x)
\label{eq:EvP_indep}
\end{equation}
where $D_x$ denotes the gradient with respect to the $x$ variables, while $\partial \RR$ is the subdifferential of $\RR$, that we recall being defined as
\begin{equation*}
\partial \RR (\bar x)=\left\{\xi\in X^* \colon \RR (x)\geq \RR (\bar x)+\scal{\xi}{x-\bar x} \quad \text{for every $x\in X$}\right\}
\end{equation*}

 We say that a continuous function $x\colon [0,T]\to X$, differentiable for almost all $t\in(0,T)$, is a \emph{solution} of the problem \eqref{eq:EvP_indep} with starting point $x_0$ if $x(0)=x_0$ and  $x(t)$  satisfies \eqref{eq:EvP_indep} for almost all $t\in(0,T)$. 

Since $\RR$ is positively homogeneous of degree one, we have that $\partial \RR(\xi)\subseteq \partial \RR(0)$, where $\partial \RR(0)$ is convex and compact, by the convexity and coercivity of $\RR$. We immediately see that 
\begin{equation*}
-D_x\EE(t,x)\in \partial \RR(\dot x(t))\subseteq \partial \RR(0)
\end{equation*}
and in particular that the initial point $x_0$ must satisfy
\begin{equation}
-D_x\EE(0,x_0)\in  \partial \RR(0)=\C
\label{eq:RISadmisgen}
\end{equation}
We will call such initial points $x_0$ \emph{admissible} for the problem.

The evolution of systems of the form \eqref{eq:EvP_indep} is widely studied; in particular we recall the following classic result (cf.~\cite[Theorem 2.1]{Mie05}).
\begin{theorem}\label{th:mielke}
	Let us consider the evolution of a system \eqref{eq:EvP_indep}, where $\EE$ is of the form \eqref{eq:energy}, with $\A\colon X \to X^*$ symmetric positive-semidefinite, $\ell(t)\colon [0,T]\to X^*$ continuously differentiable, and $\RR\colon X\to [0,+\infty]$  proper, convex, lower semi-continuous and positively homogeneous of degree one. For every initial point $x_0$ satisfying \eqref{eq:RISadmisgen} there exists at least a solution $x\in\Clip([0,T],X)$, with $x(0)=x_0$, that satisfies \eqref{eq:EvP_indep} for almost all $t\in [0,T]$.\\
		Moreover, if $\A$ is positive-definite, then such solution is unique.
\end{theorem}

As Theorem \ref{th:mielke} shows, our problem admits always at least a solution for every admissible initial point. However nothing can be said about its uniqueness: indeed,  as we show below in Example \ref{ex:2legs}, it is easy to construct example where we have multiple solutions. Such eventuality is clearly not suitable from a physical point of view. To exclude these situations and guarantee the uniqueness of the solution, we require an additional  condition on the dissipation potential.

\subsection{Uniqueness of the solution}

We now  state our uniqueness assumption. \emph{
	\begin{enumerate}[label=\textup{(\textasteriskcentered)}]
		\item For every $\tilde z\in Z$ there exists a unique $\tilde x\in X$ such that $\sigma(\tilde x)=\tilde z$ and \label{cond:symmbreak}
		\begin{equation*}
		\RR(x)-\RR(\tilde x)\geq 0 \qquad \text{for every $x\in X$ with $\sigma(x)=\tilde z$}
		\end{equation*}
\end{enumerate}}

We will discuss with more detail the nature of this assumption later and case by case.
Let us highlight here a simple consequence of \ref{cond:symmbreak}. For $\sigma (x)=0$, this assumption  tells us that moving the crawler without changing its shape dissipates a non-null energy: this assure us that, in some way, friction is present in all direction.

\begin{theorem}\label{th:indep}
	In the framework of Section \ref{sec:sub_indep},	given $\ell \in \CC^1([0,T],X)$ and an admissible initial point $x_0$, assuming that the uniqueness condition \ref{cond:symmbreak} is satisfied,  there exists a unique solution $x\in\Clip([0,T],X)$ of the evolution problem \eqref{eq:EvP_indep} that satisfies the initial condition $x(0)=x_0$.
\end{theorem}

\begin{proof}

The idea of the proof is to show that \eqref{eq:EvP_indep} can be solved by restricting to a suitable evolution problem on the shape space $Z$, satisfying all the conditions of Theorem \ref{th:mielke} including those for uniqueness of the solution, and that the evolution in the net-translation space $Y$ can be recovered uniquely by that on the shape space $Z$.

Let us recall that,  for given external load $\ell (t)$, the evolution \eqref{eq:EvP_indep} of our system can be expressed as a solution of the variational inequality
\begin{equation}
\scal{\A x(t)-\ell (t)}{u-\dot x(t)}+\RR(u)-\RR(\dot x(t))\geq 0 \quad \text{for every $u\in X$} 
\label{eq:VI}
\end{equation}
(cf.\ for instance \cite{Mie04}).
In particular this must hold when $\sigma(u)=\sigma (\dot x(t))$: for this class of functions $u$ the inequality becomes
\begin{equation}\label{xeq:dissipation_estimate}
\RR\left(\chi\bigl(\sigma(\dot x(t)),v\bigr) \right)-\RR(\dot x(t))\geq 0 \qquad \text{for every $v\in Y$}
\end{equation}
meaning that $\dot x(t)$ minimizes $\RR$ among the class of displacements $u\in X$ such that $\sigma(u)=\sigma(\dot x(t))$.

We notice that, if the uniqueness condition \ref{cond:symmbreak} holds, then the time derivative of the net translation $\dot y(t)$ can be univocally recovered as a function of that of the shape change $\dot z(t)$, namely
\begin{equation} 
\dot y(t)=\vm (\dot z(t)):=\argmin_{v\in Y}\ \RR(\chi(\dot z(t),v))
\label{eq:vm_indep}
\end{equation}
We remark that the function $\vm\colon Z\to Y$, thus defined, is positively homogeneous of degree one. 

We can use the notion of $\vm$ to reduce the dimension of the problem associated to the variational inequality \eqref{eq:VI}, using also the fact that the energy depends only on the shape component. Indeed, $x(t)$ is a solution of \eqref{eq:VI} only if $z(t)=\sigma(x(t))$ satisfies:
\begin{equation}
\scal{\Ash z(t)-\lsh (t)}{w-\dot z(t)}+\Dsh (w)-\Dsh (\dot z(t))\geq 0 \quad \text{for every $w\in Z$} 
\label{eq:RVI}
\end{equation}
where $\Dsh $ is the \lq\lq shape-restricted\rq\rq\ dissipation, i.e.~the dissipation after minimization with respect to translations of the crawler, defined as
\begin{equation} \label{eq:Dsh_indep}
\Dsh (w)=\inf_{v\in Y} \RR\bigl(\chi(w,v)\bigr)
\end{equation}
We observe that, due to assumption \ref{cond:symmbreak}, we can write $\Dsh$ as
\begin{equation} \label{xeq:Dsh_minimized}
\Dsh (w)=\RR\bigl(\chi(w,\vm(w))\bigr)
\end{equation}
This allows us to study the system for the shape changes alone and then recover the displacement $y(t)$ of the crawler through the relationship \eqref{eq:vm_indep}.

The properties of this new dissipation functional are enounced in the following lemma, that we state under slightly weaker assumptions than \ref{cond:symmbreak}, and prove after the conclusion of this proof, at page \pageref{proof_lemmaA}.
\begin{lemma} \label{lem:propDsh}
	Let us assume that, for every $w\in Z$, the function $\RR(\chi(w,\cdot))\colon Y\to \R$ is coercive.  Then $\Dsh$ is convex, lower semi-continuous and positively homogeneous of degree one.
\end{lemma}
The hypotheses of the Lemma imply the existence of a function $\vm\colon Z\to Y$,  positively homogeneous of degree one, such that the shape-restricted dissipation function $\Dsh $ of \eqref{eq:Dsh_indep} can be written in the form \eqref{xeq:Dsh_minimized}. Yet, the uniqueness of such a function $\vm$ is not required, as in \ref{cond:symmbreak}. 

\medbreak
Observing that, by the Lemma, the potential $\Dsh$ has all the desired properties, we can restate the shape-restricted variational inequality \eqref{eq:RVI} in the \emph{subdifferential formulation} of the problem, namely
\begin{equation} \label{eq:SF}
0\in \partial \Dsh (\dot z(t))+D_z\Esh(t,z(t)) 
\end{equation}
 In this case, the admissibility condition for the initial point becomes
\begin{equation}
-D_z\Esh(0,z_0)\in \partial\Dsh(0):=\Csh
\end{equation}
We observe that this condition is automatically satisfied if \eqref{eq:RISadmisgen} is true.  Indeed, we have the following characterization of $\Csh$.

\begin{lemma} \label{lemma:Csh_immersion}
 Given $\zeta \in Z^*$, let us denote with $\hat \zeta$ its canonical identification in $X^*$, namely the element $\hat \zeta\in X^*$ such that  $ \langle\hat \zeta, x\rangle=\scal{\zeta}{\sigma (x)} $ for every $x\in X$.
 Similarly, we define the subspace $\hat Z^*\subset X^*$ and the set $\hCsh \subset \hat Z^*$ as the canonical identifications of $Z^*$ and $\Csh$.
 Then $\hCsh = \C\cap \hat Z^*$.
\end{lemma}
We remark that  this characterization does not require \ref{cond:symmbreak} and the same assumptions of Lemma \ref{lem:propDsh} are sufficient.
We postpone the proof of the Lemma at the end of this section, at page \pageref{proof_lemmaB}.

We have therefore shown that, provided \ref{cond:symmbreak}, we can pass from our original problem \eqref{eq:EvP_indep}, where the invariance for translations of the energy $\EE$ does not assure the uniqueness of the solution, to the shape-restricted problem \eqref{eq:SF}, for which the energy $\Esh$ is definite-positive. Thus, we can apply Theorem \ref{th:mielke} and obtain that there exists an unique solution $z(t)$ of shape evolution problem \eqref{eq:SF}. 

Let us now consider any solution $\bar x(t)=\chi (\bar z, \bar y)$ of \eqref{eq:EvP_indep}: Theorem \ref{th:mielke} assures that there exists at least one. We have shown that $\bar z$ must coincide with the unique solution of \eqref{eq:SF}. We also know that $\dot{\bar y}$ is defined almost everywhere and integrable, being the derivative of a Lipschitz function. 
Combining this fact with \eqref{eq:vm_indep}, we deduce that also the global displacement of the crawler $\bar y(t)$ is uniquely defined as
\begin{equation}
	\bar y(t):=y_0 + \int_{0}^{t} \vm(\dot{\bar{z}}(s)) \dd s
\end{equation}
where $y_0=\pi (x_0)$.  This shows that, given \ref{cond:symmbreak}, the solution of \eqref{eq:EvP_indep} is unique.
\end{proof}


\begin{proof}[Proof of Lemma \ref{lem:propDsh}]\label{proof_lemmaA}
	Let us recall that $w\mapsto \vm(w)$ is positively homogeneous of degree one and that $\chi$ is linear. Hence, for $\lambda>0$
	\begin{align*}
	\Dsh (\lambda w)&=\RR\bigl(\chi(\lambda w, \vm(\lambda w))\bigr)=\RR\bigl( \chi(\lambda w, \lambda \vm(w))\bigr) 
	=\lambda\RR\bigl( \chi( w, \vm(w))\bigr) \\
	&=\lambda \Dsh (w)
	\end{align*}
	Thus, $\Dsh $ is positively homogeneous of degree one.
	
	To show that it is lower semi-continuous, let us suppose by contradiction that there exists a sequence $z_n\to z$ in $Z$ such that $\lim \Dsh (z_n)<\Dsh (z)$. Let us denote $y_n=\vm(z_n)$. We can always take a subsequence $y_{n_k}$ such that either $y_{n_k}\to y\in \R^d$, or $\norm{y_{n_k}}\to+\infty$ and $\left( y_{n_k}/\norm{y_{n_k}}\right) \to y_\infty\in \R^d\setminus\{0\}$.
 If $y_{n_k}\to y\in \R^d$, then, by the lower semi-continuity of $\RR$, we obtain
	$$
	\Dsh(z) \leq \RR (\chi(z,y)) \leq \lim_k \RR (\chi(z_{n_k},y_{n_k})) =\lim_k \Dsh (z_{n_k})=\lim_n \Dsh (z_{n})
	$$
	that contradicts $\lim \Dsh (z_{n})<\Dsh (z)$. If instead $\norm{y_{n_k}}\to+\infty$, let us introduce the normalized sequence
	$$
	x_k=\frac{\chi(z_{n_k},y_{n_k})}{\norm{\chi(z_{n_k},y_{n_k})}}\in X
	$$
	We notice that $\norm{\chi(z_{n_k},y_{n_k})}\to +\infty$ and $x_k\to \chi (0,y_\infty)$. Thus, recalling that $\RR$ is positively homogeneous,
	$$
	\RR (\chi (0,y_\infty)) \leq \liminf \RR(x_k)
	\leq \liminf \frac{\RR (\chi(z_{n_k},y_{n_k}))}{\norm{\chi(z_{n_k},y_{n_k})}}
	\leq \liminf \frac{\Dsh (z_{n_k})}{\norm{\chi(z_{n_k},y_{n_k})}} = 0
	$$
	since we assumed $\Dsh (z_{n_k})$ bounded. We notice that the restricted potential  $\RR (\chi (0,\cdot))$ is not only convex and coercive, but also positively homogeneous of degree one, and therefore has an unique minimum in $y=0$. This is in contradiction with $\RR (\chi (0,y_\infty))\leq 0$; we can therefore deduce that $\Dsh$ is lower semi-continuous. 
	
	To prove the convexity of $\Dsh $, we observe that for every  $0\leq \lambda \leq 1$, writing $w_\lambda=\lambda w + (1-\lambda \bar w)$, we have
	\begin{align}
	\lambda \Dsh (w)+(1-\lambda)\Dsh  (\bar w) &\geq 
	\lambda \RR\bigl(\chi (w,\vm(w))\bigr) +(1-\lambda)\RR \bigl(\chi (\bar w,\vm(\bar w))\bigr) \notag\\
	&\geq \RR\bigl(\chi (w_\lambda, \lambda\vm(w)+(1-\lambda)\vm(\bar w))\bigr) \notag\\
	&\geq \RR\bigl(\chi (w_\lambda, \vm(w_\lambda))\bigr) =\Dsh (w_\lambda)
	\end{align}	
\end{proof}

\begin{remark}[Stasis domains] \label{rem:SD}
	The reformulation of the problem in the proof of Theorem \ref{th:indep} shows that the set $\hCsh$ has a key role in the description of the problem. First of all we observe that 
$$
-D_x\EE(t,x)\in \hCsh
$$
where the vector on the left hand-side provides a description of the tension along the crawler. Unlike with $\C$, all the states in $\hCsh$ can actually be attained by the tension vector (for a suitable choice of $\ell$). Moreover, equation \eqref{eq:SF} implies that all the changes in the tension of the crawler within the interior (with respect to $\hat Z^*$)	of $\hCsh$ occur always without any movement of the crawler. Thus, to move, the crawler must necessarily reach one of the tension states at the boundary of $\hCsh$ (with respect to $\hat Z^*$). 
For this reason we refer to $\Csh$ as \emph{stasis domain} of the crawler \cite{GidDeS16}, in analogy to the elastic domains in elastoplasticity. 
	
\end{remark}

\begin{remark} \label{rem:lipschitz}
	We remark that in Theorems \ref{th:mielke} and \ref{th:indep} it is sufficient to require $\ell(t)$ to be only continuous and piecewice continuosly differentiable: in such a case we can apply iteratively the theorem to each time interval in which $\ell\in \CC^1$, picking as initial condition the value of the solution at the end of the previous interval. Such a weaker requirement allows to consider, for instance, a triangle wave input, cf.~\cite{DeSGidNos15}.
\end{remark}

Let us recall here some useful facts  (cf.\ for instance \cite{RocWets}). For  a proper, lower semicontinuous, and
convex functional $F\colon V\to \R\cup \{+\infty\}$, where $V$ is a Banach space, we can define the 
\emph{Legendre--Fenchel transform} $F^*\colon V^*\to \R\cup \{+\infty\}$ of $F$ as
\begin{equation}
F^*(\xi)=\sup\{\scal{\xi}{v}-F(v): v\in V\}
\end{equation}
The functional $F^*$ is is proper, lower semi-continuous and convex; moreover we have the following equivalence for subdifferentials
\begin{equation}\label{eq:LegendreFenchel}
\xi \in \partial F(v)  \qquad \text{if and only if} \qquad v\in \partial F^*(\xi)
\end{equation}

Now we apply this fact to the evolution problem \eqref{eq:EvP_indep}.  

Since the functional $\RR$ is positively homogeneous, we can characterize it in terms of $\C$. A generalized Euler identity (cf.~\cite{YangWei08}) tells us that
\begin{equation}
	\RR(x)=\scal{\xi}{x} \qquad \text{for every $\xi \in \partial \RR(x)$} 
\end{equation}
Since $\partial \RR(x)\subseteq \partial \RR(0)$, we deduce that
\begin{equation} \label{eq:R_fenchelchar}
	\RR(x)=\max_{\xi\in \C}\scal{\xi}{x}
\end{equation}
As a consequence of the notion of normal cone, we have that
\begin{equation} \label{xeq:normal_char}
	\RR(x)=\scal{\xi}{x} \qquad \text{if and only if} \qquad x\in \NN_\C(\xi)
\end{equation}
where $\NN_\C (\xi)$ denotes the normal cone with respect to the set $\C$ in the point $\xi$. We recall that, for a convex set $\C\in X^*$, we define the normal cone $\NN_\C (\xi)\subset X$ as the set of $x\in X$ such that $\langle\tilde\xi-\xi,x\rangle\leq 0$ for every $\tilde \xi\in \C$.

Moreover, we know that
\begin{equation*}
\RR^*(\xi)=I_\C(\xi) \qquad \partial\RR^*(\xi)=\partial I_\C(\xi)=\NN_\C (\xi)
\end{equation*}
Here, we denoted with $I_\C$ the characteristic function of the set $\C$ (i.e.~ the function having value zero if $\xi\in\C$ and $+\infty$ elsewhere)

These facts can be used to obtain a first characterization of condition \ref{cond:symmbreak} and to associate the compatible direction of motion to the boundary points of a stasis domain. 
\begin{lemma} \label{lemma:SB_normalcone}	
	Let us assume that, for every $w\in Z$, the function $\RR(\chi(w,\cdot))\colon Y\to \R$ is coercive. Then, condition \ref{cond:symmbreak} is equivalent to the following:
for every $\xi \in \hCsh$ and every $x_1,x_2\in \NN_{\C}(\xi)$, we have	$\sigma(x_1)=\sigma(x_2)$ if and only if $x_1=x_2$.
\end{lemma}
\begin{proof}
	Let us take $x_1,x_2\in \NN_{\C}(\xi)$ with	$\sigma(x_1)=\sigma(x_2)$ and $\xi \in \hCsh$. To prove that assumption \ref{cond:symmbreak} implies the condition of the Lemma, we have to show that it implies $x_1=x_2$.	
	Since $\xi\in \hat Z^*$, we have $\RR(x_1)=\scal{\xi}{x_1}=\scal{\xi}{x_2}=\RR(x_2)$. Moreover, since $\xi \in \hCsh$, we have $\scal{\xi}{x_1}\leq \Dsh(\sigma(x_1))$ and so $\RR(x_1)=\Dsh(\sigma(x_1))$. By \ref{cond:symmbreak}, this is true only if $x_1=x_2$.
	
	To show that \ref{cond:symmbreak} is also necessary for the condition of the Lemma, we assume by contradiction that \ref{cond:symmbreak} is falsified by $x_1\neq x_2$. Let us notice that for every $x_1\in X$ with  $\RR(x_1)=\Dsh(\sigma(x_1))$, we can always find $\xi\in \hCsh$ such that $\scal{\xi}{x_1}=\RR(x_1)$. To do so, we take any $\zeta\in \partial \Dsh (\sigma(x_1))$ and set $\xi=\hat \zeta$. By $\eqref{xeq:normal_char}$, we have $x_1\in \NN_\C(\xi)$. Since $\sigma(x_1)=\sigma(x_2)$ and $\xi\in \hat Z^*$, we have $\scal{\xi}{x_2}=\RR(x_1)=\RR(x_2)$ and therefore, by $\eqref{xeq:normal_char}$, also $x_2\in \NN_\C(\xi)$, contradicting in this way the condition of the Lemma.
\end{proof}

\begin{remark}[Stasis domains and direction of motion]\label{rem:SDdirection}
	By \eqref{eq:LegendreFenchel} and noticing that $-D_x\EE(t,x)=\ell(t)-\A x$, we know that the evolution problem \eqref{eq:EvP_indep} is equivalent to
	\begin{equation} \label{eq:pre_sweeping}
	\dot x(t)\in \partial I_\C(-D_x\EE(t,x))=\NN_\C(\ell(t)-\A x)
	\end{equation}
	Equation \eqref{eq:pre_sweeping} shows which directions of motion are compatible with a with a ``tension'' state $\hat w=\ell(t)-\A x\in \hCsh$. Those corresponds to the projection of $\NN_\C(\hat w)$ on the net translation subspace $\hat Y^*$. If $\hat w$ is in the interior of the stasis domain $\hCsh$, it is also in the interior of $\C$, so the normal cone is $\NN_\C(\hat w)=\{0\}$ and the crawler cannot move, as we already now. If $\hat w$ is on the boundary of $\hCsh$, then $\NN_\C(\hat w)$ contains at least a line. The projection of the cone $\NN_\C(\hat w)$ on $\hat Y^*$ provides the directions of possible net translation attainable with a suitable input.
\end{remark}

To conclude, we prove Lemma \ref{lemma:Csh_immersion}.

\begin{proof}[Proof of Lemma \ref{lemma:Csh_immersion}] \label{proof_lemmaB}
	Let us first show that $\hCsh \subseteq \C\cap \hat Z^*$. By construction we have $\hCsh \subseteq \hat Z^*$. Moreover, for every $\zeta \in \Csh$ and $x\in X$, applying the characterization \eqref{eq:R_fenchelchar} to $\Dsh$ and $\Csh$, we have
	$$
	\bigl\langle \hat \zeta ,x\bigr\rangle=\scal{\zeta}{\sigma(x)}\leq \Dsh(\sigma (x))\leq \RR(x)
	$$
	and therefore $\hat \zeta\in \C$.
	
	To show the opposite inclusion, let us take any $\hat \zeta \in \C\cap \hat Z^*$; we remark that since $\hat \zeta \in \hat Z^*$ there exists a corresponding $\zeta \in Z^*$. Then, for every $z\in Z$, we have
		$$
	\scal{\zeta}{z}=\scal{\hat \zeta}{\chi(z,\vm(z))}\leq \RR(\chi(z,\vm(z)))=
	\Dsh(z)
	$$
	and therefore $\zeta \in \hCsh $.
\end{proof}


\section{Time-independent friction: examples}
\label{sec:indep_examples}

We now illustrate how the abstract framework of the previous section can be easily applied to crawling locomotors. We remark that, when not explicitly stated otherwise, the term ``crawling'' here usually refers to locomotion with a unidimensional body, moving on the line.
In the study of crawling locomotion we can recognize two main families of models: discrete models, where the body of is represented by a finite number of contact points, and continuous models, where the friction is distributed along a 
domain (e.g.~a segment). Models of both kinds has been extensively studied, with various choices in the morphology of the crawler and in the frictional interaction. For the discrete case, we mention here the series of works
 \cite{BPZZ11,BPZZ16,SteBeh12,ZZB09}, dealing with chains of blocks subject to a peristaltic wave, followed by a recent focus on adaptive control for such situations \cite{BSZZB17,Behn2011,Behn2013}. 
For the continuous case, we point out the papers \cite{DeSGNT13,Tanaka_2012} for peristaltic waves, and \cite{DeSTat12,GidNosDeS14} for both peristaltic waves and actuation via two shape-parameters.

In this section we consider both discrete and continuous models. We first discuss homogeneous models, meaning that friction and elastic coefficients are the same along all the body. Indeed, this situation is common in both biological and robotic crawlers, since a modular structure is convenient during construction/growth and provides the locomotor with redundancy, useful in case of damage. In Section \ref{sec:hybrid} we show, however, that the same approach can be extended also to heterogeneous models and to more complex body structures, combining both a discrete and a continuous part. In Section \ref{sec:planarcrawler} we briefly discuss the case of planar crawlers ($d=2$).

\subsection{Homogeneous discrete models} \label{sec:disc_indep}

In discrete models  we represent the body of the crawler with a finite number of points $\Omega_N=\{\xi_1,\xi_2,\dots,\xi_N\}\subset\R$, with $N\geq 2$. Such a system is represented in Figure \ref{fig:discr} for $N=3$.  

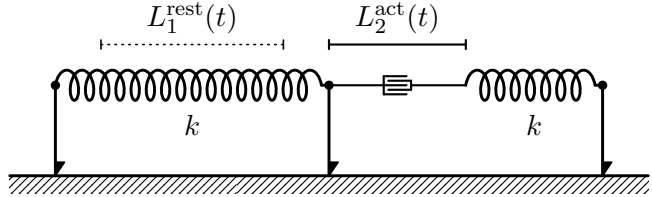
\begin{figure}
	\begin{center}

\begin{tikzpicture}[line cap=round,line join=round,>=triangle 45,x=1.0cm,y=1.0cm, line width=1.1pt,scale=1.2]
\clip(0.,-0.3) rectangle (8.,2.);
\draw[dotted, line width=0.6pt] (1.5,1.45)-- (3.5,1.45);
\draw[line width=0.6pt] (1.5,1.4)-- (1.5,1.5);
\draw[line width=0.6pt] (3.5,1.4)-- (3.5,1.5);

\draw (2.5,1.45) node[anchor=south] {$L_1^\mathrm{rest}(t)$};
\draw (2.5,0.8) node[anchor=north] {$k$};
\draw (1.,1.)-- (1.,0.);
\draw[decoration={aspect=0.5, segment length=2mm, amplitude=2mm,coil},decorate] (1,1.)-- (4.,1.);
\draw (4.,1.)-- (4.,0.);
\draw (0.5,0.)-- (7.5,0.);
\draw [fill=black] (1.,1.) circle (1pt);
\draw [fill=black,line width=0.6pt] (1.,0.) --  (1,0.15) -- (1.1,0.15) --(1,0);
\draw [fill=black,line width=0.6pt] (4.,0.) --  (4,0.15) -- (4.1,0.15) --(4,0);
\draw [fill=black] (4.,1.) circle (1pt);
\fill [pattern = north east lines] (0.5,0) rectangle (7.5,-0.2);

\draw[line width=0.8pt] (4.,1.)-- (4.6,1.);
\draw[line width=0.8pt] (4.9,1.)-- (5.5,1.);
\draw[line width=0.8pt] (4.6,0.9)-- (4.6,1.1);
\draw[line width=0.8pt] (4.9,0.95)-- (4.9,1.05);
\draw[line width=0.8pt] (4.65,0.95)-- (4.9,0.95);
\draw[line width=0.8pt] (4.65,1.05)-- (4.9,1.05);
\draw[line width=0.8pt] (4.6,1)-- (4.85,1);
\draw[line width=0.8pt] (4.6,1.1)-- (4.85,1.1);
\draw[line width=0.8pt] (4.6,0.9)-- (4.85,0.9);
\draw[line width=0.8pt] (4,1.45)-- (5.5,1.45);
\draw[line width=0.8pt] (4,1.4)-- (4,1.5);
\draw[line width=0.8pt] (5.5,1.4)-- (5.5,1.5);
\draw (4.75,1.45) node[anchor=south] {$L_2^\mathrm{act}(t)$};
\draw (6.25,0.8) node[anchor=north] {$k$};
\draw [decoration={aspect=0.5, segment length=2mm, amplitude=2mm,coil},decorate](5.5,1.)-- (7.,1.);
\draw (7.,1.)-- (7.,0.);
\draw [fill=black,line width=0.6pt] (7.,0.) --  (7,0.15) -- (7.1,0.15) --(7,0);
\draw [fill=black] (7.,1.) circle (1pt);
\end{tikzpicture}
	\end{center}
	\caption{A discrete crawler with three contact points. The picture illustrates the two possible interpretation of the energy \eqref{eq:linkenergy} of each link. The left link is a spring subject to an active distortion that changes its rest length $L_1^\mathrm{rest}(t)$ in time. The right link is composed by an actuator, with variable length $L_2^\mathrm{act}(t)$, coupled with a spring.}
	\label{fig:discr}
\end{figure}

The displacement of the crawler can is described by a vector $x=(x_1,x_2,\dots,x_N)$ in $X=\R^{N}$, identifying with $x_i$ the displacement of the point $\xi_i$. Since we work in a finite dimensional space, to simplify the notation we will adopt the canonical identification of $(\R^N)^*$ with $\R^N$, so that the dual coupling $\scal{\cdot}{\cdot}$ corresponds to the scalar product.

To describe the changes in shape, we look at the differences between the displacements $x_i$,$x_{i+1}$ of each couple of adjacent points $\xi_i$, $\xi_{i+1}$. Thus we set $Z=\R^{N-1}$ and
\begin{equation}
\sigma (x)=(x_2-x_1, x_3-x_2,\dots, x_N-x_{N-1}):=(z_1,z_2,\dots, z_{N-1})
\end{equation}
We point out that this base of $Z$ is not orthogonal with respect to the scalar product induced by $X$. This means, for instance, that $\Csh\subset Z^*$ may appear deformed with respect to its immersion $\hCsh$ in $X^*$, as happens in Figure \ref{fig:3legsdomains}.

To measure the net displacement of the crawler, several choices are equivalently valid. 
For instance, we can represent the net translation $y$ of the crawler by that of any of its points, such as its head ($\pi(x)=x_N$) or tail ($\pi(x)=x_1$). In \cite{GidDeS17}, the model with $N=3$ has been analysed taking $\pi(x)=x_2$.  
As these examples show, the immersion $\hat Y$ in $X$ of net translation space $Y$ may not be orthogonal to the immersion $\hat Z$ of the shape space $Z$, with respect to the scalar product of $X$. 
An orthogonal decomposition remains anyway in general a good choice. To do so, let us choose as reference point the barycentre of the crawler, namely 
\begin{equation}
\pi(x)=\frac{1}{N}\scal{\mathds{1}_N}{x}=\frac{1}{N}\sum_{i=1}^N x_i:=y
\end{equation}
where $\mathds{1}_N$ denotes the vector in $\R^N$ with all components equal to $1$.

We assume that the body of the crawler has $N-1$ elastic links joining each couple of consecutive blocks. 
We assume that the energy of each link, joining $\xi_i$ with $\xi_{i+1}$, has the form
\begin{equation}\label{eq:linkenergy}
\EE_i(t,x)=\frac{k}{2}(x_{i+1}-x_i-L_i(t))^2
\end{equation}
Hence, the internal energy $\EE(t,x)$ of the crawler is the sum of the $N-1$ terms $\EE_i(t,x)$ associated with each link.
The terms $L_i\in \CC^1([0,T],R)$ represent our control -- mediated by elasticity -- on the shape of the crawler (more generally, $L_i(t)$ can be taken continuous and piecewise continuously differentiable, cf.~Remark~\ref{rem:lipschitz}). The form \eqref{eq:linkenergy} of the energy of each link  has two possible interpretations, represented in Figure \ref{fig:discr}.

The first one is to model the elastic link as a spring with elastic constant $k$, subject to an active distortion that changes its rest length in time, with $L_i^\mathrm{rest}(t)=\xi_{i+1}-\xi_i+L_i(t)$. The second interpretation is to assume that the segment is composted by two adjacent elements: a spring with elastic modulus $k$ and rest length $L_i^\mathrm{rest}$, and an actuator which changes its length $L_i^\mathrm{act}(t)$ in time. In this case we have $L_i(t)=-(\xi_{i+1}-\xi_i)+L^\mathrm{rest} + L^\mathrm{act}(t)$. 
 Here, for simplicity, we assume that all the $N-1$ springs share the same constant. 
Yet, the model falls within our abstract framework also without such an assumption, as discussed  in Section \ref{sec:hybrid}. 
 
We have
$$
\EE(t,z) = \sum_{i=1}^{N-1}\frac{k}{2}(x_{i+1}-x_i-L_i(t))^2
$$
that can be expressed in the form \eqref{eq:energy} by taking
$$
\A=\frac{k}{2}\,\begin{pmatrix} 
1  & -1 & 0 & \cdots & 0 & 0\\
-1 &  2 & -1& \cdots & 0 & 0\\
0  & -1 & 2 & \cdots & 0 & 0\\
\vdots & \vdots & \vdots & \ddots & \vdots & \vdots \\
0 & 0 & 0 & \cdots & 2 & -1\\
0 & 0 & 0 & \cdots & -1 & 1\\
\end{pmatrix}
\qquad
\ell(t)=k\,\begin{pmatrix}
-L_1(t)\\
-L_2(t)+L_1(t)\\
-L_3(t)+L_2(t)\\
\vdots\\
-L_{N-1}(t)+L_{N-2}(t)\\
L_{N-1}(t)
\end{pmatrix}
$$
Since the internal energy $\EE$ depends only on shape, it can be rewritten in the  coordinates $(z_1,\dots,z_{n-1})$, as
\begin{equation}
\Esh(t,z) = \sum_{i=1}^{N-1} \frac{k}{2} \left( z_i-L_i(t)\right)^2  
\end{equation}
where $\Ash$ and $\lsh$ have the simpler form
$$
\Ash=\frac{k}{2} \I_{N-1}
\qquad
\lsh(t)=k\,\begin{pmatrix}
L_1(t)\\
L_2(t)\\
\vdots\\
L_{N-1}(t)
\end{pmatrix}
$$

We assume that every element $\xi_i$ is subject to the same type of directional dry friction, so that the force exerted on it by the surface is
\begin{equation}
	F_i(t)=F(\dot{x}_i(t)) \quad \text{ where } F(v) \in
	\begin{cases}
	\displaystyle \{\mu_-\}   &  \text{if $v<0$} \\
	\displaystyle [-\mu_+,\mu_-]   &   \text{if $v=0$} \\
	\displaystyle \{-\mu_+\}   &  \text{if $v>0$}
	\end{cases}
\end{equation}
where $\mu_\pm$ are the (non-negative) friction coefficients.
Thus we can define the dissipation potential associated to each element $\xi_i$ as
\begin{equation} \label{xeq:dissipoint}
	\RR_i(u_i)=\begin{cases}
	-\mu_-u_i  &  \text{if $u_i\leq 0$} \\
	\mu_+u_i  &  \text{if $u_i\geq 0$}
	\end{cases}
\end{equation}
and the total dissipation potential as
\begin{equation}
	\RR(\dot x(t))=\sum_{i=1}^N \RR_i (\dot x_i(t))
\end{equation}
We observe that $\RR$ is continuous, convex and positively homogeneous of degree one, since it is the sum of functions with these same properties. We have
$$
\C=[-\mu_-,\mu_+]^N\subseteq X^*
$$
thus establishing the set of admissible initial points.

We can now investigate the nature of condition \ref{cond:symmbreak} for this family of models.

\begin{lemma} \label{lemma:SBdiscr}
	For a homogeneous discrete crawler, each of the following conditions is equivalent to \ref{cond:symmbreak}:
	\begin{enumerate}[label=\textit{\roman*)}]
		\item \label{it:SBdiscr1} For every integer $m\in \{0,1,\dots,N\}$, we have
		$$
		m\mu_- - (N-m)\mu_+ \neq 0
		$$
		\item \label{it:SBdiscr2} The hyperplane $\hat Z^*=\{x\in \R^N : \scal{x}{\mathds{1}_N}=0\}$ does not intersect any of the vertices of the set $\C$; in other words, the vertices of $\C$ are not in $\hCsh$.
	\end{enumerate}
\end{lemma}

\begin{proof}
	Condition \ref{cond:symmbreak} can be restated saying that, for any $\tilde z \in Z$, the restriction of the function $\RR$  to the line $\chi(\R,\tilde z)\subset X$ has a unique minimum point. Since the function $\RR$ is convex, so are its restrictions to affine spaces and therefore the set of minimizers on each of these restrictions is always a interval, possibly empty.   If a restriction $\RR(\chi(w,\cdot))$, for some $w\in Z$ has no minimum, then it is no coercive, and so must be $\RR$. But $\RR$ is not coercive only if $\mu_-=0$ or $\mu_+=0$, contradicting \ref{it:SBdiscr1}.

	If the set of minimizers is non-empty and not a singleton, then there exists a minimizer $\tilde x\in \chi(\R,\tilde z)$ and $\Lambda>0$ such that 
	$$
	\RR(\tilde x)=\RR (\tilde x+\lambda \mathds{1}_{\!N}) \qquad \text{for every $0\leq \lambda\leq \Lambda$}
	$$
	where we recall $\mathds{1}_{\!N}=(1,1,\dots,1)\in \R^N$.
	Let us denote with $m(x)$ the number of strictly negative components $x_i$ of $x$. Without loss of generality, we can take $\Lambda$ sufficiently small that $m(\tilde{x})=m(\tilde{x}+\Lambda \mathds{1})$.
	Thus, for every $0\leq \lambda\leq \Lambda$, we have
	$$
	\RR (\tilde x+\lambda \mathds{1}_{\!N}) -\RR(\tilde x)= \lambda [m(\tilde x)\mu_- -(N-m(\tilde{x}))\mu_+]=0 
	$$
	and therefore $m(\tilde x)\mu_-=(N-m(\tilde{x}))\mu_+$. 
This shows that, if \ref{it:SBdiscr1} is satisfied, then the set of minimizers must be a singleton, hence \ref{cond:symmbreak} holds.
	
 To check the converse implication, let us assume by contradiction that there  exists $m\in \{0,1,\dots,N\}$ such that $m\mu_- = (N-m)\mu_+$. We observe that, for $\lambda\in [0,1]$, the vectors
	$$
	x^\lambda=(\overbrace{-\lambda,-\lambda,\dots,-\lambda}^{\text{$m$-times}},
	\overbrace{1-\lambda,1-\lambda,\dots,1-\lambda}^{\text{$(N-m)$-times}})\in X
	$$
	have all the same projection on $Z$, namely the vector $\tilde z$ with $\tilde{z}_{m-1}=1$ and the other components equal to zero. Furthermore, since we assumed $m\mu_- = (N-m)\mu_+$ and $\RR$ is convex, we have that all the $x_\lambda$ are minimizers for $\RR$ restricted to $\chi(\R,\tilde z)$, in contradiction with \ref{cond:symmbreak}. We have thus shown that \ref{it:SBdiscr1} if also necessary for \ref{cond:symmbreak}, and therefore the two conditions are equivalent.
	
	\medbreak
	
To show the equivalence between \ref{it:SBdiscr1} and \ref{it:SBdiscr2}, let us consider a generic vertex $x\in X$ of the cube $\C$: it has $m$ coordinates with value $-\mu_-$ and $N-m$ coordinates with value $\mu_+$, where $m\in \{0,1,\dots,N\}$. Such a vertex lies on the hyperplane $\hat Z^*$ if and only if 
		$$
	\scal{x}{\mathds{1}_N}= -m\mu_- + (N-m)\mu_+ = 0
	$$
	since $\mathds{1}_N$ is orthogonal to $\hat Z^*$.
For any $m\in \{0,1,\dots,N\}$ we can find a corresponding vertex, so it follows that \ref{it:SBdiscr1} and \ref{it:SBdiscr2} are equivalent.
\end{proof} 

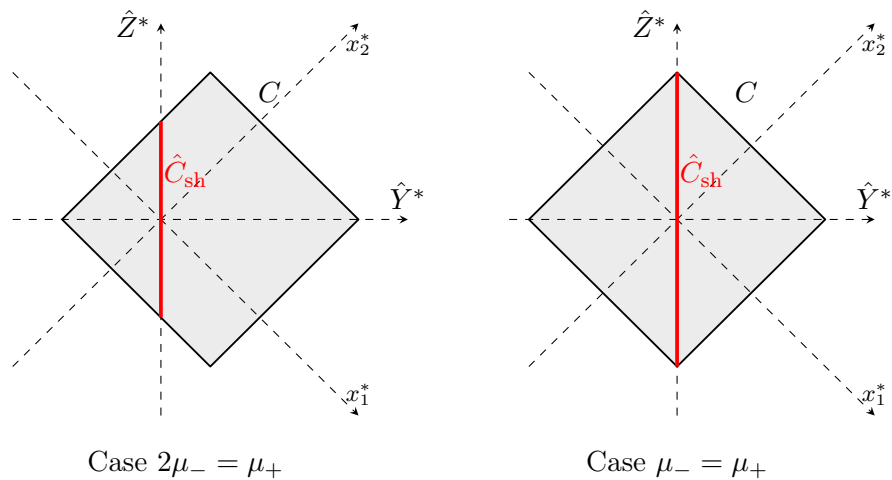
\begin{figure}
	\begin{center}
		\begin{tabular}{cc}
			\begin{tikzpicture}[line join=round,>=stealth,x=1.3cm,y=1.3cm]
			\clip(-2.2,-2.2) rectangle (2.7,2.2);
			\draw[->,dashed, color=black] (-1.5,-1.5) -- (2,2) node[below]{\footnotesize $x^*_2$};
			\draw[->,dashed, color=black] (-1.5,1.5) -- (2,-2) node[above]{\footnotesize $x^*_1$};
			\draw[->,dashed, color=black] (-1.5,0.) -- (2.5,0.) node[above]{ $\hat Y^*$};
			\draw[->,dashed,color=black] (0.,-2) -- (0.,2) node[left]{ $\hat Z^*$};
			\fill[line width=0.7pt,draw=black,fill=uuuuuu,fill opacity=0.1] (0.5,1.5) -- (-1.,0.) -- (0.5,-1.5) -- (2.,0.) -- cycle;
			\draw [line width=1.4pt,color=red] (0.,1.)-- (0.,-1.);
			\draw (1.1,1.3) node{$\C$};
			\draw (0.25,0.5) node[color=red]{$\hCsh$};
			\end{tikzpicture}
			&
				\begin{tikzpicture}[line join=round,>=stealth,x=1.3cm,y=1.3cm]
			\clip(-2.2,-2.2) rectangle (2.2,2.2);
			\draw[->,dashed, color=black] (-1.7,0.) -- (2,0.) node[above]{ $\hat Y^*$};
			\draw[->,dashed,color=black] (0.,-2) -- (0.,2) node[left]{ $\hat Z^*$};
				\draw[->,dashed, color=black] (-1.5,-1.5) -- (2,2) node[below]{\footnotesize $x^*_2$};
			\draw[->,dashed, color=black] (-1.5,1.5) -- (2,-2) node[above]{\footnotesize $x^*_1$};
			\fill[line width=0.7pt,draw=black,fill=uuuuuu,fill opacity=0.1] (0.,1.5) -- (-1.5,0.) -- (0.,-1.5) -- (1.5,0.) -- cycle;
			\draw [line width=1.4pt,color=red] (0.,1.5)-- (0.,-1.5);
			\draw (0.7,1.3) node{$\C$};
			\draw (0.25,0.5) node[color=red]{$\hCsh$};
			\end{tikzpicture}\\
			Case $2\mu_-=\mu_+$ &
			Case $\mu_-=\mu_+$
		\end{tabular}
	\end{center}
	\caption{Representation of the sets $\C$ and $\hCsh$ in the space $X^*$ for Example \ref{ex:2legs}. The case $2\mu_-=\mu_+$, represents the general qualitative scenario, while the case $\mu_-=\mu_+$ is critical, allowing multiple solutions.}
	\label{fig:2legsdomains}
\end{figure}

\begin{example}\label{ex:2legs}
	The case of a homogeneous discrete crawler with two contact points, $\Omega=\{\xi_1,\xi_2\}$ has been considered in \cite{DeSGidNos15};  such model is similar to the prototype presented in \cite{NosDeS14}, and should be compared with two-anchor crawlers, for which shape change and friction manipulation occur at separate time phases (cf.~Section \ref{sec:dep_examples}).
	
	For this model we have
	$$
	\C=[-\mu_-,\mu_+]^2\subset X^* \qquad \Csh=[-\mum,\mum]\in Z^*
	$$
	This situation is illustrated in Figure \ref{fig:2legsdomains} (notice that the change of coordinates $\sigma$ induces a rescaling of a factor $\sqrt{2}$ between $\Csh$ and $\hCsh$). By Remark \ref{rem:SDdirection} we see at a glance that in the case $2\mu_-=\mu_+$ the crawler is able only to move backwards.
	
	Condition \ref{cond:symmbreak} reduces to assume that the friction coefficients are non-zero and $\mu_+\neq \mu_-$. We point out that the case $\mu_+=\mu_-$ is not meaningful as a crawler, since, as discussed in the introduction, it lacks the necessary asymmetry to achieve a locomotion.
	
	The motility problem in the case when $L_1(t)$ is a triangle wave, corresponding to a periodic cycle of contraction and elongation of the crawler, has been explicitly solved in \cite{DeSGidNos15}, and is qualitatively shown in Figure \ref{fig:2motion}.  The crawler is able to advance only in the direction of lower friction $\mum$, and will do it only if the oscillations of $L_1$ are sufficiently large. More precisely, the amplitude $\Delta L$ must be greater than $2\mum/k$ to periodically produce displacement,  and in that case on each period the crawler will move of $\Delta L-2\mum/k$.

	\medbreak
	
	This basic model provides an opportunity to illustrate concretely how the failure of condition \ref{cond:symmbreak} can produce multiplicity of solutions. Let us therefore assume that \ref{cond:symmbreak} is false, namely $\mu=\mu_+=\mu_-$.
	We consider the following situation. At the initial time $t=0$, the body of the crawler is in the state of maximum admissible compression, namely
	$$
	-k[x_2-x_1-L_1(t)]=-D_z\Esh(0,z_0)=\mu
	$$
	The load evolves as $L_1(t)=L_1(0)+t$. If the crawler would remain steady, this would further increase the tension of the crawler. Yet, since we are at the boundary of the stasis domain and we have reached the maximum compression, the crawler would start elongating in such a way that $-D_z\Esh(t,z(t))=\mu$. This implies that the distance $z(t)$ between the two contact points grows as $L_1$, namely $z(t)=z(0)+t$. By the shape reduction approach, we know that the solution of the systems are given by the functions $x\colon [0,T]\to \R^2$ such that $x_2(t)-x_1(t)=z(t)$ and producing minimal dissipation, as written in \eqref{xeq:dissipation_estimate}. These conditions are satisfied by all the the functions of the form
	$$
	x_1(t)=x_1(0)-\lambda t \qquad x_2(t)=x_2(0)+(1-\lambda)t
	$$
	with $\lambda\in[0,1]$. In other words, the interval $[-1/2,1/2]$ is the set of minimizers for $\RR(\chi(1,\cdot))$. For $\lambda =0$ this means that $\xi_1$ moves backwards while $\xi_2$ is steady; in the opposite case $\lambda =1$, the point $\xi_1$ is steady while $\xi_2$ moves forwards. In the intermediate cases $\lambda\in(0,1)$ both contact points move, $\xi_1$ backwards and $\xi_2$ forwards.
\end{example}

\begin{example}\label{example:3legs}
		We now consider the case of a homogeneous discrete crawler with three contact points, $\Omega=\{\xi_1,\xi_2,\xi_3\}$. This model has been studied in \cite{GidDeS16}, to which we refer for a detailed analysis of the motility. 
		
		Here, condition \ref{cond:symmbreak} requires that the friction coefficients are non-zero, and satisfy $\mu_+\neq 2\mu_-$ and $\mu_-\neq 2\mu_+$. Qualitatively, these assumptions create the divide between two different scenarios.
		
\begin{itemize}
	\item If $\mu_+> 2\mu_-$	or	$\mu_-> 2\mu_+$, then the crawler is able to move  only in the direction of lower friction.
	\item If $\frac{1}{2}\mu_-<\mu_+< 2\mu_-$, then it is possible to find suitable periodic inputs $L$ that allow the crawler to move in each direction.
\end{itemize}
		
		Such qualitative difference in the motility is reflected by the shape of the stasis domains $\Csh$, illustrated in Figure \ref{fig:3legsdomains}.		
\end{example}

\begin{figure}[tb!]
\begin{center}
	\begin{tabular}{cc}
	\includegraphics[width=6cm]{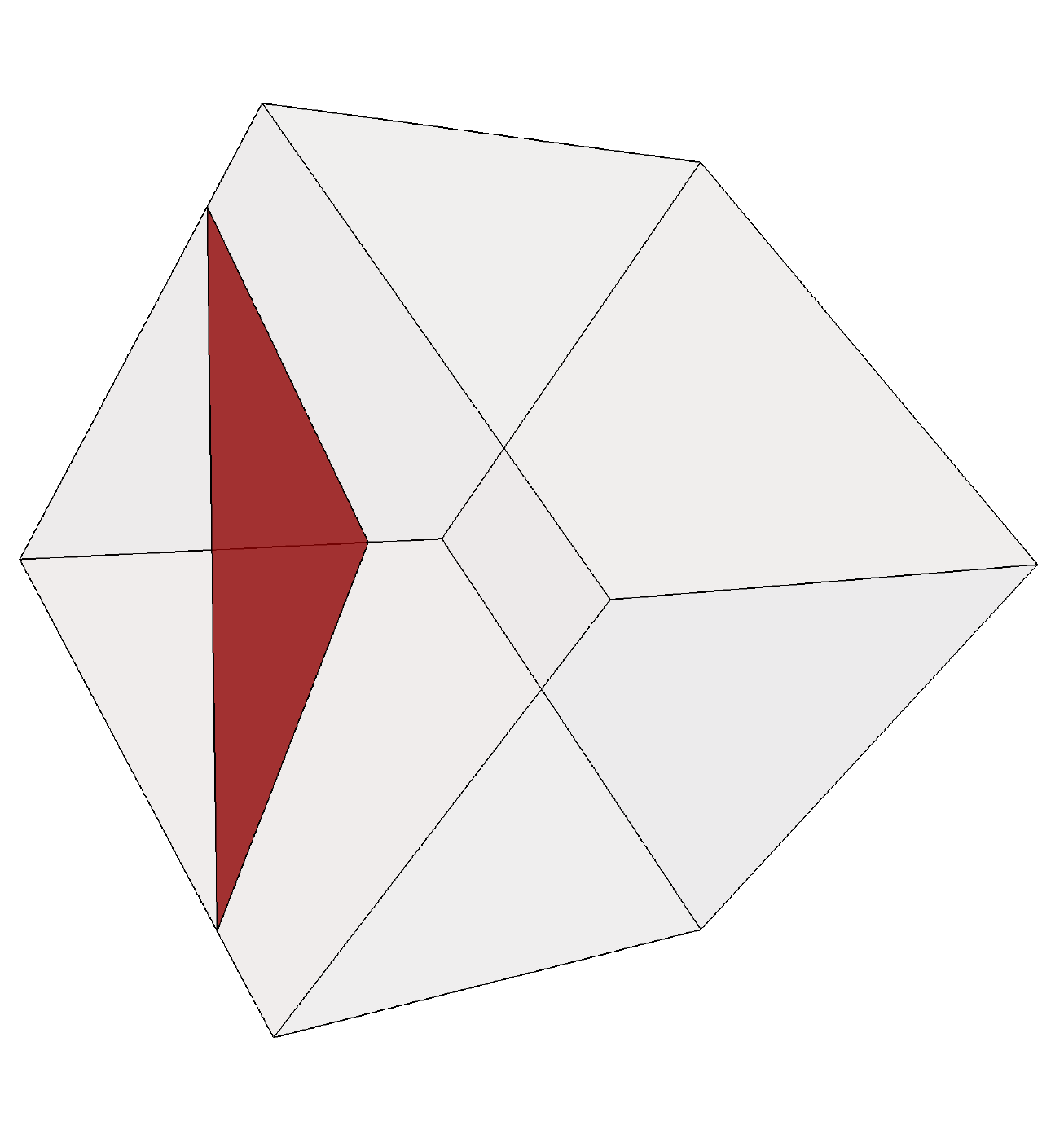}	& 	\includegraphics[width=6cm]{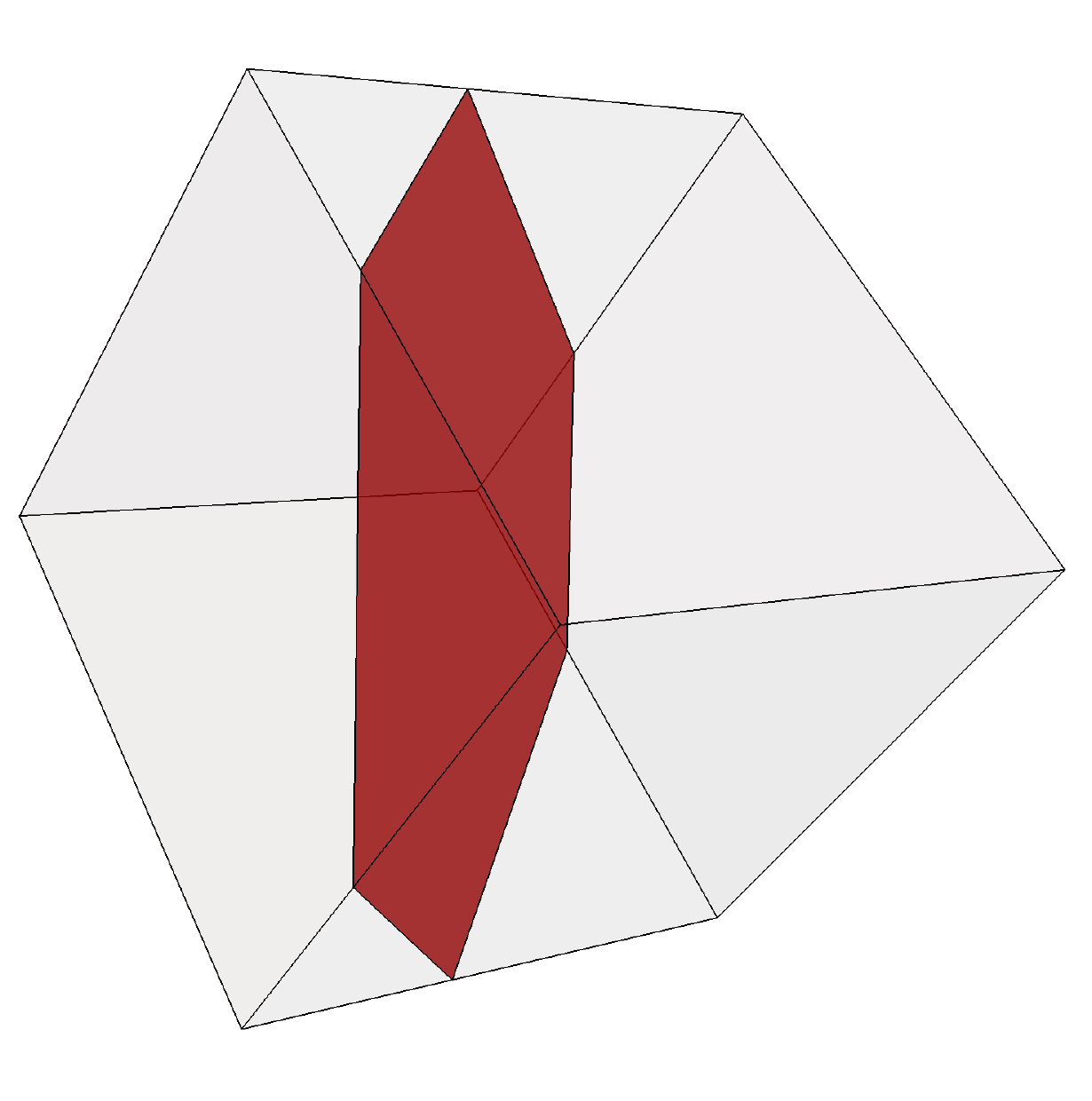} \\
		\begin{tikzpicture}[line cap=round,line join=round,x=1.0cm,y=1.0cm,>=stealth,scale=1]
		\clip(-3,-2.5) rectangle (3.5,2.5);\footnotesize
		\draw[->,dashed, color=black,thin] (-1.5,0.) -- (2.5,0.) node[above]{ $z^*_1$};
		\draw[->,dashed,color=black,thin] (0.,-2.3) -- (0.,2) node[left]{ $ z^*_2$};
		\fill[line width=1 pt,color=red,fill=red,fill opacity=0.2] (-1.,1.) -- (2.,1.) -- (-1.,-2.) -- cycle;
		\draw [] (0.6,1) node[anchor=south] {$\vm\geq 0$};
		\draw [](-1,-0.5) node[anchor= east] {$\vm\geq 0$};
		\draw [](0.5,-0.5)node[anchor=north west] {$\vm\geq 0$};
		\draw [](0.4,0.4)node {\normalsize$\Csh$};
		\draw [line width=1.2pt] (-1.,1.)-- (2.,1.);
		\draw [line width=1.2pt] (2.,1.)-- (-1.,-2.);
		\draw [line width=1.2pt] (-1.,-2.)-- (-1.,1.);
		\end{tikzpicture} 
		&
		\begin{tikzpicture}[line cap=round,line join=round,x=1.cm,y=1.0cm,>=stealth,scale=1]
		\clip(-3.,-2.5) rectangle (3.,2.5);\footnotesize
			\draw[->,dashed, color=black,thin] (-1.5,0.) -- (2.5,0.) node[above]{ $z^*_1$};
		\draw[->,dashed,color=black,thin] (0.,-2.3) -- (0.,2) node[left]{ $ z^*_2$};
		\fill[line width=1.2pt,color=red,fill=red,fill opacity=0.2] (-1.,-1.333) -- (-0.3333,-1.333) -- (1.333,0.3333) -- (1.333,1.) -- (-0.3333,1.) -- (-1.,0.3333) -- cycle;
		\draw [line width=1.2pt] (-1.,-1.333) -- (-0.3333,-1.333) -- (1.333,0.3333) -- (1.333,1.) -- (-0.3333,1.) -- (-1.,0.3333)-- cycle;		
		\draw  (0.6,1) node[anchor=south] {$\vm\geq 0$};
		\draw  (-1,-0.5) node[anchor=east] {$\vm\geq 0$};
		\draw  (0.5,-0.5)  node[anchor=north west] {$\vm\geq 0$};
		\draw  (-0.6666,-1.333) node[anchor=north] {$\vm\leq 0$};
		\draw  (1.333,0.6666)  node[anchor=west] {$\vm\leq 0$};
		\draw (-0.6666,0.6666)  node[anchor=south east] {$\vm\leq 0$};
		\draw [](0.4,0.4)node {\normalsize$\Csh$};
		\end{tikzpicture}\\
		$\mu_->2\mu_+$
		&
		$\mu_+<\mu_-<2\mu_+$
	\end{tabular}
\end{center}
\caption{Representation of the sets $\C\subset X^*$, with in red the section $\hCsh$ (above), and $\Csh\subset Z^*$ (below) for Example \ref{example:3legs}, in the two main qualitative scenarios (cf.\cite{GidDeS16}). The compatible direction of motion, indicated for the edges of the stasis domains, can be obtained at a glance from the pictures above, applying Remark \ref{rem:SDdirection}.}
\label{fig:3legsdomains}
\end{figure}
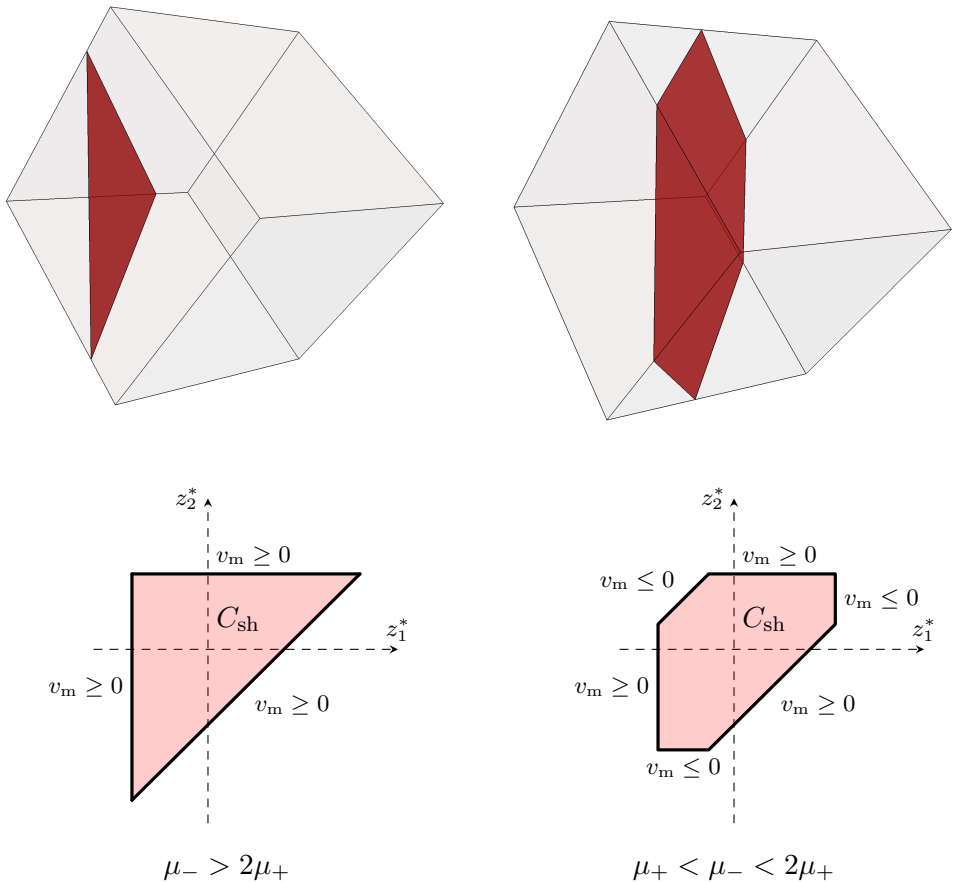

\subsection{Interpretation as sweeping process} 
\label{sec:sweeping}

The case of homogeneous discrete models of crawler encourages a deeper exploration of the nature of condition \ref{cond:symmbreak}, by formulating the problem in the form of \emph{sweeping process}.

We recall that, since we are considering models with $X=\R^N$, we  simplify our notation by canonically identifying $X^*$ with $X$, and consequently the dual pairing $\scal{\cdot}{\cdot}$ with the scalar product in $X$. Moreover, we denote with $\hat \sigma, \hat \pi$ the immersions in $X$ of the projections $\sigma,\pi$, namely
$$
\hat \sigma (x)=\chi(\sigma(x),0)\in X \qquad 
\hat \pi (x)=\chi(0,\pi(x))\in X
$$

We recall that $-D_x\EE(t,x)=\ell(t)-\A x$, and that the evolution problem \eqref{eq:EvP_indep} is equivalent to \eqref{eq:pre_sweeping}:
\begin{equation*} 
\dot x(t)\in \partial I_\C(-D_x\EE(t,x))=\NN_\C(\ell(t)-\A x)
\end{equation*}
We observe that
\begin{align*}
\NN_\C\bigl(\ell(t)-\A x\bigr)
=\NN_{\tilde C(t)} \bigl(-\frac{k}{2}\hat \sigma(x)\bigr)
\end{align*}
where we set
\begin{equation*}
	C'(t)=\C-\ell(t)
\end{equation*}
Introducing the new variable $u=-\frac{k}{2}\,x$, our evolution problem becomes 
\begin{equation} \label{eq:sweep_projected}
	-\dot u(t)\in \NN_{C'(t)} (\hat \sigma(u))
\end{equation}
Here it is evident how the current net translation $\pi(x(t))$ of the crawler does not influence its motility.

 We continue our analysis by noticing that $\hat \sigma(u)=u-\hat \pi(u)$, so that we get
\begin{equation}\label{eq:sweep}
-\dot u(t)\in \NN_{\tilde C(t,u)} (u) 
\end{equation}
where the set $\tilde C(t)$ is defined as
\begin{equation}\label{eq:sweep_set}
\tilde C(t,u)=\C-\ell(t)+\hat{\pi}(u)
\end{equation}
With observe with \eqref{eq:sweep} that our system can be expressed as a sweeping process, but associated to a state-dependent set $\tilde C(t,u)$. 
State-dependent sweeping processes has been discussed in several papers \cite{KunMM98,KunMM2000}. Usually conditions of the form
$$
\dist \bigl(\tilde C(t,u),\tilde C(s,v)\bigr)\leq L_1\abs{t-s}+L_2 \norm{u-v}
$$
with $L_1>0$ and $0<L_2<1$ are required for existence/uniqueness. In our case the condition above would satisfied for $L_2=1$ but not for a smaller one, so the situation is different.

Let us now focus on the structure available in our setup.
 First of all, the set $\tilde C(t,u)$ is always a translation of a constant set $\C$; this will no longer be true in Section \ref{sec:dep_theory} where we introduce time-dependence in the dissipation potential. The two vectors $\ell(t)$ and $\hat \pi(u)$ that characterize the translations are orthogonal. The vector $\ell(t)$, that is our input or control on the system, is $(N-1)$-dimensional and translates the set $\C$ in the shape subspace $\hat Z^*$. The term $\hat \pi(u)$ is $1$-dimensional and moves the set in the net translation subspace $\hat Y^*$; this is done by keeping the state of the system always in the section of the set corresponding to $\hCsh$. Namely we have
 \begin{equation} \label{eq:sweep_section}
 	 u(t)\in \hCsh -\ell(t)+\hat \pi(u(t)) \quad \text{for every $t\in [0,T]$}
 \end{equation}
We can interpret this fact in the following way. The state dependence $\pi(u(t))$ could be theoretically replaced by a time dependent term $p(t)$. 
Oppositely to $\ell(t)$, that is a  free input, the term $p(t)$ is controlled by the system itself, with the aim to keep \eqref{eq:sweep_section} satisfied.

Clearly we need some condition to assure us that such a choice of a suitable $p(t)$ is unique: this is done by condition \ref{cond:symmbreak}. In this sense, its characterization given in Lemma \ref{lemma:SB_normalcone} is really useful. For instance, if for $\xi\in \hCsh$ there where two distinct vectors $w_1,w_2\in \NN_\C(\xi)$ with $\sigma(w_1)=\sigma(w_2)$, the problem with initial point $u(0)=\xi$ and input $\ell(t)=\sigma(w_1)t$ would have as a solution $u(t)=\xi-w_i t$, for both $i=1,2$.

\medbreak
 Our sweeping process interpretation of crawling locomotion, presented above, is illustrated in Figure \ref{fig:sweeping} for the discrete crawler with two contact points, described in Example~\ref{ex:2legs}.

\begin{figure}[p]
	\begin{center}
		\begin{tikzpicture}[line cap=round,line join=round,>=triangle 45,x=1.0cm,y=1.0cm, scale=0.15]
		\clip(-5.,-8.) rectangle (78.,42.);\footnotesize
		\draw[thick, orange] (0.,24.)-- (16.,34.)-- (32.,24.)-- (48.,34.)-- (64.,24.) node[anchor=west] {$L(t)$};
		\draw[thick, green] (0.,13.5)-- (2.,13.5)-- (16.,10.)-- (20.,10.)-- (32.,7.)-- (36.,7.)-- (48.,4.)-- (52.,4.)-- (64.,1.) node[anchor=west] {$y(t)+\frac{\xi_1+\xi_2}{2}$};		
		\draw[thick, blue] (64.,-1.) node[anchor=north west] {$x_1(t)+\xi_1$}-- (48.,-1.)-- (36.,5.)-- (16.,5.)-- (2.,11.)-- (0.,11.);
		\draw [thick, blue](64.,3.) node[anchor=south west] {$x_2(t)+\xi_2$}-- (52.,9.)-- (32.,9.)-- (20.,16.)-- (0.,16.);
		\draw [dotted] (43.,-5.)-- (43.,38.) node[anchor=south]{\cirlet{i}};
		\draw [dotted] (36.,-5.)-- (36.,38.) node[anchor=south]{\cirlet{h}};
		\draw [dotted] (32.,-5.)-- (32.,38.) node[anchor=south]{\cirlet{g}};
		\draw [dotted] (26.,-5.)-- (26.,38.) node[anchor=south]{\cirlet{f}};
		\draw [dotted] (20.,-5.)-- (20.,38.) node[anchor=south]{\cirlet{e}};
		\draw [dotted] (16.,-5.)-- (16.,38.) node[anchor=south]{\cirlet{d}};
		\draw [dotted] (9.,-5.)-- (9.,38.) node[anchor=south]{\cirlet{c}};
		\draw [dotted] (2.,-5.)-- (2.,38.) node[anchor=south]{\cirlet{b}};
		\draw [dotted] (0.,-5.)-- (0.,38.) node[anchor=south]{\cirlet{a}};
		\draw[->] (-2.,-5.)-- (65.,-5.) node[anchor=north]{$t$};
		\draw[->] (-2.,-5.)-- (-2.,20.);
		\draw[->] (-2.,20.)-- (-2.,40.);
		\draw[->] (-2.,20.)-- (65.,20.)node[anchor=north]{$t$};
		\end{tikzpicture}
	\end{center}
	\caption{Motility of the crawler with two contact point of Example \ref{ex:2legs}, with $\mu_-<\mu_+$. The marked times refer to Figure \ref{fig:sweeping}, illustrating the same situation in the sweeping process interpretation.}
	\label{fig:2motion}
\end{figure}
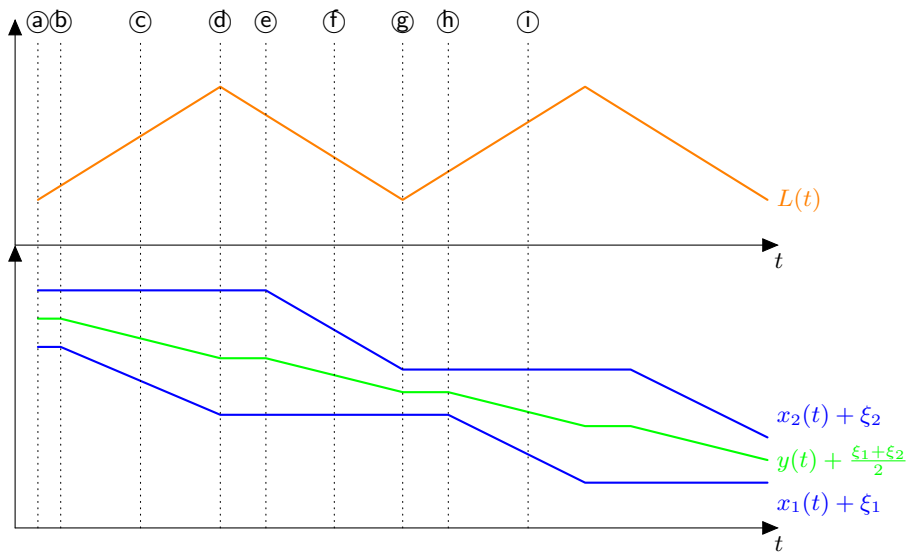

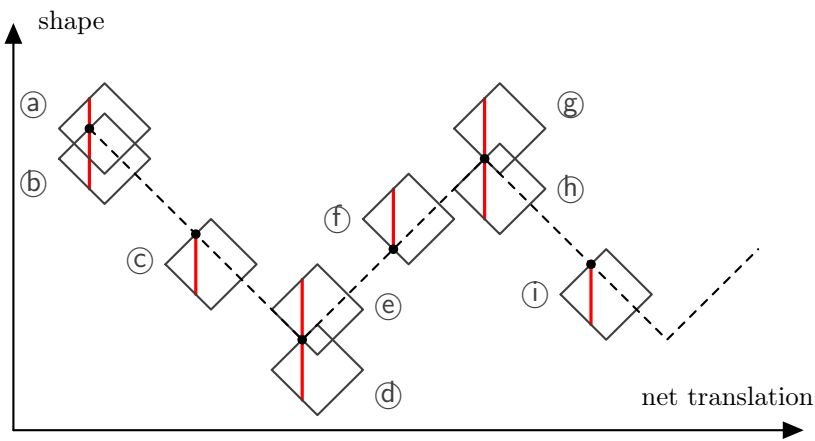
\begin{figure}[p]
	\begin{center}

\begin{tikzpicture}[line cap=round,line join=round,>=triangle 45,x=1.0cm,y=1.0cm,scale=0.2,thick]
\clip(-10.,-19.) rectangle (47.,11.);
\draw [color=red,  very thick] (-2.,4.)--(-2.,-2.);
\draw [color=red,  very thick] (12.,-8.)--(12.,-16.);
\draw [color=red, very thick] (24.,4.)--(24.,-4.);
\draw [color=red, very thick] (5.,-5.)--(5.,-9.);
\draw [color=red, very thick] (18.,-6.)--(18.,-2.);
\draw [color=red, very thick] (31.,-7.)--(31.,-11.);
\draw[color=uuuuuu] (-4.,2.) node[anchor=south east]{\cirlet{a}}-- (-1.,5.) -- (2.,2.) -- (-1.,-1.) -- cycle ;
\draw [color=uuuuuu] (-4.,0.)-- (-1.,3.)-- (2.,0.)-- (-1.,-3.)-- (-4.,0.) node[anchor=north east]{\cirlet{b}};
\draw [color=uuuuuu] (3.,-7.)-- (6.,-4.)-- (9.,-7.)-- (6.,-10.)-- (3.,-7.) node[anchor=east]{\cirlet{c}};
\draw [color=uuuuuu] (10.,-14.)-- (13.,-11.)-- (16.,-14.) node[anchor=north west]{\cirlet{d}}-- (13.,-17.)-- (10.,-14.);
\draw [color=uuuuuu] (10.,-10.)-- (13.,-7.)-- (16.,-10.) node[anchor=west]{\cirlet{e}}-- (13.,-13.)-- (10.,-10.);
\draw [color=uuuuuu] (16.,-4.)-- (19.,-1.)-- (22.,-4.)-- (19.,-7.)-- (16.,-4.) node[anchor=east]{\cirlet{f}};
\draw [color=uuuuuu] (22.,2.)-- (25.,5.)--(28.,2.) node[anchor=south west]{\cirlet{g}}-- (25.,-1.)-- (22.,2.);
\draw [color=uuuuuu] (22.,-2.)-- (25.,1.)-- (28.,-2.) node[anchor=west]{\cirlet{h}}-- (25.,-5.)-- (22.,-2.);
\draw [color=uuuuuu] (29.,-9.)-- (32.,-6.)-- (35.,-9.)-- (32.,-12.)-- (29.,-9.) node[anchor=east]{\cirlet{i}};
\draw [thick, dashed](-2.,2.)-- (12.,-12.) --(24.,0.)--(36.,-12.)-- (42.,-6.);

\draw [fill=black] (-2.,2.) circle (7pt);
\draw [fill=black] (12.,-12.) circle (7pt);
\draw [fill=black] (24.,0.) circle (7pt);
\draw [fill=black] (5.,-5.) circle (7pt);
\draw [fill=black] (18.,-6.) circle (7pt);
\draw [fill=black] (31.,-7.) circle (7pt);
\draw [->](-7,-18)--(-7,9);
\draw [->](-7,-18)--(45,-18);\small
\draw (-6,9)node[anchor=west] {shape};
\draw (40,-17)node[anchor=south] {net translation};
\end{tikzpicture}
		
\end{center}
\caption{The evolution of the crawler with two contact points of Example \ref{ex:2legs}, expressed in the sweeping process formulation.  In the picture above, the set $C(t,u)$ is represented at several times, corresponding to those of Figure \ref{fig:2motion}. The red sections are the subsets satisfying \eqref{eq:sweep_section}, corresponding to the section $\hCsh$ of the domain $\C$. We remark that the change of variable in \eqref{eq:sweep_projected} has inverted the orientation, so that the rightward motion of the sweeping process corresponds to a leftward motion of the crawler; the control on the shape subspace is analogously inverted.}
\label{fig:sweeping}
\end{figure}

\medbreak

Finally, we remark that, starting from the shape-restricted problem \eqref{eq:SF}, it is possible to obtain a \lq\lq classical\rq\rq\ sweeping process governed by a set dependent only on time and not on the current state.
Yet, the purpose of this short excursion in the sweeping process formulation is to illustrate, once again, the nature and role of the uniqueness condition \ref{cond:symmbreak}, and to make clearer the reverse route of the shape-reduction, showing how a $(N-1)$-dimensional input is used to control a $N$-dimensional system.

\subsection{Homogeneous continuous models} \label{sec:cont_indep}
We now consider continuous models of crawlers. 
Let us therefore assume that the body is represented  by a compact interval $\Omega=[\xi_a,\xi_b]\subset \R$. To describe the space of displacements $X$, we adopt the Sobolev space $X=W^{1,2}(\Omega, \R)$. As shape space $Z$ we choose  $W^{1,2}_0(\Omega, \R)$; the net translation of a crawler with displacement $x$ can be therefore naturally identified with its average (as discussed in the discrete case, we remark that this is only one of several possible parametrizations).
With the notation above, this means
\begin{equation}
	\pi(x) =\frac{1}{\abs{\Omega}}\int_{\Omega}x(\xi)\dd \xi \in Y=\R
\end{equation}
and
\begin{equation}
\sigma(x)=x-\pi(x) \in W^{1,2}_0(\Omega, \R)
\end{equation}
We set the internal energy of the crawler as
\begin{equation}
	\EE(t,x)= \frac{k}{2}\int_{\Omega}\left( x'(\xi)-\epsilon(t,\xi)\right)^2 \dd \xi
\end{equation}
The term $\epsilon(t,\xi)$, which represents our input on the system, is the \emph{active distortion}, corresponds to the spontaneous strain of the point $\xi$ of the body at time $t$.
Such model applies  well to active materials such as nematic elastomers~\cite{DeSGidNos15}.

The internal energy $\EE(t,x)$ can be written in the form \eqref{eq:energy} by taking
\begin{align} \label{xeq:Al_uniform_cont}
\scal{\A x}{x}= \frac{k}{2}\int_{\Omega}\left( x'(\xi)\right)^2 \dd \xi &&
\scal{\ell(t)}{x}=k\int_{\Omega} x'(\xi)\epsilon(t,\xi) \dd \xi
\end{align}
We observe that the operator $\A$ is positive semi-definite on $W^{1,2}(\Omega,\R)$ and its kernel is the space of constant functions.

As in the discrete case, we assume that the same  anisotropic  dry friction law is acting along all the body; since we have distributed interactions, we consider the frictional force per unit reference length
\begin{equation}
	f(\xi,t)=f(\dot{x}(\xi,t))\quad \text{ where }f(v)\in
	\begin{cases}
	\displaystyle \{\mud_-\}  & \text{ if $v<0$ }  \\
	\displaystyle [-\mud_+,\mud_-]  &  \text{ if $v=0$}  \\
	\displaystyle \{-\mud_+\}  & \text{ if $v>0$ } 
	\end{cases}
\end{equation}
where $\mud_+\pm$ are non-negative coefficients. 
	For a function $x(\xi)$, let us denote with $x^+$ and $x^-$ respectively its positive and negative part, namely
$$
x^+(\xi)=\begin{cases}
x(\xi) & \text{if $x(\xi)\geq 0$}\\ 0 & \text{if $x(\xi)\leq 0$}
\end{cases}
\qquad\qquad 
x^-(\xi)=\begin{cases}
0& \text{if $x(\xi)\geq 0$}\\ x(\xi) & \text{if $x(\xi)\leq 0$}
\end{cases}
$$
so that $x(\xi)=x^+(\xi)+x^-(\xi)$
We  define the dissipation potential as
\begin{equation} \label{eq:defRcont_omog}
	\RR(\dot x(\cdot,t))=\int_{\Omega} -\mud_- \dot x^-(\xi,t)+\mud_+ \dot x^+(\xi,t) \dd \xi
\end{equation}

We observe that the dissipation potential $\RR$ thus defined is convex, continuous  and positively homogeneous of degree one.  The admissibility condition \eqref{eq:RISadmisgen} becomes
\begin{equation}
	 -k x'(\xi)\in [-\mud_-,\mud_+] \qquad \text{for a.e.~$\xi\in\Omega$}
\end{equation}

We now look at the interpretation of condition \ref{cond:symmbreak} in this framework.

\begin{lemma} For a homogeneous continuous crawler, condition \ref{cond:symmbreak} is satisfied if and only if the coefficients $\mud_+,\mud_-$ are strictly positive. 
\label{lemma:SBcont}
\end{lemma}

\begin{proof}
	First of all, we observe that condition \ref{cond:symmbreak} is equivalent to the uniqueness of the minimum point for each of the restrictions of $\RR$ to the line $\chi(\tilde z,\R)\subset X$, with  $\tilde z \in Z$.
	
 We notice that if the coefficients are both strictly positive, then $\RR$ is coercive and therefore each restriction has at least one minimum point. Let us assume that such a minimum is not unique. Then, by the convexity of $\RR$, there exists a minimizer $\tilde x\in X$ and a positive constant $\Lambda>0$ such that
\begin{equation}\label{eq:lemma_cont_omog}
		\RR(\tilde x)=\RR(\tilde x +\lambda) \qquad \text{for every $\lambda\in[0,\Lambda]$}
\end{equation}
	Thus, by the definition of $\RR$ in \eqref{eq:defRcont_omog}, we can compute, for every $\bar \lambda\in(0,\Lambda)$, the right derivative
	\begin{align*}
	\lim_{\lambda \to \bar \lambda^+} \RR (\tilde x +\lambda)&=
	\mud_+ \meas \left[  \tilde{x}^{-1}\bigl([-\bar \lambda,+\infty)\bigr)\right] -
	\mud_- \meas \left[ \tilde{x}^{-1}\bigl((-\infty, -\bar \lambda)\bigr)\right]  =0
	\end{align*}

We observe that such a condition can be satisfied only if one of the coefficients $\mud_\pm$ is equal to zero. Indeed, if they where both strictly positive, then we would have, for every $\bar \lambda\in(0,\Lambda)$
\begin{equation}\label{eq:lemma_cont_omog2}
\frac{\mud_+}{\mud_-}	 \meas \left[  \tilde{x}^{-1}\bigl([-\bar \lambda,+\infty)\bigr)\right] =
\meas \left[ \tilde{x}^{-1}\bigl((-\infty, -\bar \lambda)\bigr)\right]
\end{equation}
requiring both pre-images to have positive measure, meaning that $(-\Lambda,0)$ is included in the image of $\tilde{x}$. Since $\tilde{x}$ is continuous, this implies that the two measures of the pre-images are both strictly monotone with respect to $\bar{\lambda}\in (0,\Lambda)$; but  one decreasing while the other one increasing. This means that \eqref{eq:lemma_cont_omog2} can be true for at most one $\bar{\lambda}\in (0,\Lambda)$, in contradiction with \eqref{eq:lemma_cont_omog}. 
Hence having strictly positive coefficients $\mud_\pm$ is a sufficient condition for \ref{cond:symmbreak}.

To show that such a condition is also necessary, let us consider the restriction of $\RR$ to the line $\chi(0,\R)\subset X$, that is exactly the set of all  constant functions on $\Omega$. Restricted to this set, $\RR$ has always a minimum in the constant zero function, with value zero. If $\mud_+=0$, then $\RR$ assumes value zero also on all the positive constant functions, that become therefore all minimizers. Analogously, if $\mud_-=0$,  all the negative constant functions  become minimizers. Hence it is necessary for \ref{cond:symmbreak} that both $\mud_\pm$ are not zero.	 
\end{proof}

\begin{example} In \cite{DeSGidNos15}, the motility of a homogeneous continuous crawler with anisotropic dry friction has been explicitly solved for a cyclic input consisting on a homogeneous contraction-extension, namely $\epsilon(t,\xi)=\epsilon_0(t)$, with $\epsilon_0(t)$ a triangular wave oscillating between $0$ and $\Delta\epsilon>0$. Indicating with  $l=\xi_b-\xi_a$ the rest length of the crawler for $\epsilon_0=0$, and assuming without loss of generality that $\mu_->\mu_+$, it is shown that on average the crawler moves only if $\Delta\epsilon>\mu_+l/k$ and the displacement produced in a cycle, in the direction of lower friction, is equal to
	$$
	l\left( \Delta\epsilon-\frac{\mu_+l}{k}\right) \frac{\mu_--\mu_+}{\mu_-+\mu_+}
	$$  
	To conclude, we observe that the case $\mu_+=\mu_-$ with a periodic contraction-extension cycle is not significant for motility, since, by basic symmetry arguments,  it  produces a  displacement with zero net translation.
	
\end{example}


\subsection{Heterogeneous  models} \label{sec:hybrid}

\paragraph{Heterogeneous elasticity}

We now generalize the discrete case of Section \ref{sec:disc_indep}, allowing the elastic constant to change from spring to spring. Moreover, we can also consider different connection structures between the elements of the body, instead of the simple chain graph used in Section \ref{sec:disc_indep}. We illustrate these ideas with the following example.

\begin{example} In a crawler with three legs, i.e.~$\Omega=\{\xi_1,\xi_2,\xi_3\}$, we can consider three springs, with positive elastic constants $k_\alpha,k_\beta,k_\gamma$ and time-dependent term $L_\alpha,L_\beta,L_\gamma$, joining respectively the couples of points $(\xi_1,\xi_2)$, $(\xi_2,\xi_3)$ and $(\xi_1,\xi_3)$. With the same notation of Section \ref{sec:disc_indep}, we obtain
$$
\A=\frac{1}{2}\,\begin{pmatrix} 
 k_\alpha+k_\gamma & - k_\alpha & -k_\gamma & \\
-k_\alpha &  k_\alpha+k_\beta & -k_\beta& \\
-k_\gamma  & -k_\beta & k_\beta+k_\gamma & \\
\end{pmatrix}
\qquad
\ell(t)=\begin{pmatrix}
-k_\alpha L_\alpha(t)-k_\gamma L_\gamma(t)\\
-k_\beta L_\beta(t)+k_\alpha L_\alpha(t)\\
k_\beta L_\beta(t) + k_\gamma L_\gamma(t)\\
\end{pmatrix}
$$
or, in the shape coordinates
$$
\Ash=\frac{1}{2}\,\begin{pmatrix} 
k_\alpha +k_\gamma & k_\gamma  \\
k_\gamma  & k_\beta+k_\gamma 
\end{pmatrix}
\qquad
\lsh(t)=\begin{pmatrix}
k_\alpha L_\alpha(t)+k_\gamma L_\gamma(t)\\
k_\beta L_\beta(t)+k_\gamma L_\gamma(t)\\
\end{pmatrix}
$$
\end{example}

Analogously, for continuous models as those of Section \ref{sec:cont_indep}, we can assume that the elastic modulus changes along the body of the crawler. Thus, for any $k\in L^\infty (\Omega,\R)$ with strictly positive values, we can set
\begin{align}  \label{eq:energ_cont}
\scal{\A x}{x}= \frac{1}{2}\int_{\Omega}k(\xi)\left( x'(\xi)\right)^2 \dd \xi &&
\scal{\ell(t)}{x}=\int_{\Omega}k(\xi) x'(\xi)\epsilon(t,\xi) \dd \xi
\end{align}
an observe that the desired properties of the operator $\A$ are preserved.

We remark that such  more general description of the internal energy of the crawler does not affect the modelling of the frictional interaction. In particular, the characterizations of condition \ref{cond:symmbreak} given in Lemmas \ref{lemma:SBdiscr} and \ref{lemma:SBcont} remain valid also in this more general framework, provided that the homogeneity of the model is preserved in the part concerning friction.

\paragraph{Heterogeneous friction}

For discrete models as those of Section \ref{sec:disc_indep}, it is possible to pick different directional friction coefficients $\mu_\pm^i$ for each element $\xi_i$ of the body. In this case the total dissipation potential is defined as
\begin{equation}\label{xeq:dissipdiscr_asy}
\RR(u)=\sum_{i=1}^N \RR_i (u_i) \qquad\text{with}\, \RR_i(u_i)=\begin{cases}
-\mu_-^iu_i  &  \text{if $u_i\leq 0$} \\
\mu_+^iu_i  &  \text{if $u_i\geq 0$}
\end{cases}
\end{equation}
Similarly, for continuous models as those of Section \ref{sec:cont_indep}, we can assume that the dissipation coefficients change along the body of the crawler, and so define define the dissipation potential as
\begin{equation} \label{xeq:dissipcont}
\RR(u)=\int_{\Omega} -\mud_-(\xi) \dot u^-(\xi)+\mud_+(\xi) \dot u^+(\xi) \dd \xi
\end{equation}
where the functions  $\mud_\pm \in L^2(\Omega,\R)$ assume non-negative values.

We observe that in both situations the dissipation potentials $\RR$ thus defined are convex, continuous  and positively homogeneous of degree one. 

This is sufficient to guarantee the existence of a solution for \eqref{eq:EvP_indep}; its uniqueness  requires an additional case-specific study of condition like \ref{cond:symmbreak}. We can still observe that a necessary condition for uniqueness is that each of the $\mu_\pm$ is not the constant zero function on $\Omega$ in the continuous case, or zero on all the point in the discrete case. Otherwise it would be possible to rigidly translate the body in the corresponding direction without dissipation (e.g.~$\RR(\chi(0,v))=0$ for every $v\in[0,+\infty)$, if $\mu_+\equiv 0$). The same argument of Lemma \ref{lemma:SBcont} can be adapted to show that having strictly positive coefficients $\mu_\pm$ is a sufficient condition for \ref{cond:symmbreak}.
In the discrete case we have the following result.

\begin{lemma} \label{lemma:SBdiscr_asym} 
For a discrete crawler with heterogeneous friction defined as in \eqref{xeq:dissipdiscr_asy}, condition \ref{cond:symmbreak} is satisfied if and only if, for every subset $J\subseteq\{1,2,\dots,N\}$ we have
\begin{equation}\label{xeq:SBdiscr_asym}
\sum_{i\in J}\mu_-^i +\sum_{i\in J^c}\mu_+^i \neq 0
\end{equation}
where $J^c=\{1,2,\dots,N\}\setminus J$.
\end{lemma}
\begin{proof}
We proceed as in the proof of Lemma \ref{lemma:SBdiscr}
and assume that the points $\tilde x+\lambda \mathds{1}_{\!N}$ are minimizers for $\RR$ for every $0\leq \lambda\leq \Lambda$.
Without loss of generality we can assume that all the the minimizers in this family have all non-zero components: we therefore denote with $J\subseteq\{1,2,\dots,N\}$ the set of indexes such that the corresponding component of the minimizers is strictly negative
	Thus, for every $0\leq \lambda\leq \Lambda$, we have a contradiction for \eqref{xeq:SBdiscr_asym}, since
	$$
	0=\RR (\tilde x+\lambda \mathds{1}_{\!N}) -\RR(\tilde x)= \lambda \left[\sum_{i\in J}\mu_-^i +\sum_{i\in J^c}\mu_+^i \right]
	$$
Hence we have shown that \eqref{xeq:SBdiscr_asym} implies \ref{cond:symmbreak}. To check the converse implication, we proceed again as done in Lemma \ref{lemma:SBdiscr}. If by contradiction there exists a set $J$ for which \eqref{xeq:SBdiscr_asym} is false, we can define, for  $\lambda\in [0,1]$, the family of vectors $x^\lambda\in \R^N$ having components $x^\lambda_i=\lambda$ if $i\in J$, while $x^\lambda_i=1-\lambda$ if $i\notin J$. Arguing as above, we see that $\RR$ is constant on these vectors $x^\lambda$, that therefore are all minimizers, in contradiction with \ref{cond:symmbreak}.
\end{proof}

\begin{remark}[Hybrid models]
Above we illustrated some basic and likely situations, yet the flexibility of the rate-independent systems framework applies to more general situations

For instance, we remark that discrete models can be seen as continuous ones, where the dissipation functional is the a sum of suitable Dirac functions. Indeed, in the discrete case we are implicitly solving the evolution problem on each elastic link, assuming a constant deformation along the element.
This step proves really useful, producing a radical simplification of the model.

With the same spirit we can consider hybrid models, having both a continuous and a discrete  part. For example, we can consider a crawler whose contact set $\Omega$ is the union of an interval and a singleton, assuming that the two elements are joined by an elastic link. As above, we can implicitly assume the linear behaviour of the link and take as state space $X=W^{1,2}([0,1],\R)\otimes \R$, with the first factor representing the displacement of the continuous part, and the other one the displacement of the isolated contact point. Internal energy and dissipation potential are defined analogously to what we have done in this section, noticing that the shape subspace $Z$ depends on how the elastic link is attached to the continuous component. 
 \end{remark}

	
\subsection{A few words on planar crawlers} \label{sec:planarcrawler}

This paper is focused on unidimensional rectilinear crawlers, namely on the case $Y=\R$. Yet the same setup and results, as Theorems \ref{th:indep} and \ref{th:dep}, can be applied also to locomotors in higher dimensional spaces. In particular we refer to crawlers on a plane, corresponding to $Y=\R^2$. This transition, however, produces a remarkable rise in the complexity of the system and the emergence of new modelling issues. Here, we briefly discuss the main points.

\begin{itemize}
	\item \emph{Complexity in deformations: shear/curvature.} Planar systems introduce a qualitatively different family of deformations, those associate to shear of bidimensional bodies, or bending in unidimensional bodies. Such deformations provide a second kind of motility strategies: indeed, changes in the natural curvature are employed in the motility of snakes and snake-like robots. On the other hand, this correspond to a considerable increase of the degrees of freedom of the system.
	\item \emph{Complexity in anisotropic friction.} For rectilinear crawlers, friction is described by only two parameters $\mu_\pm$ (possibly depending on the contact point $\xi$ of the crawler). In planar models, these two parameters have to be replaced by a function $\mu\colon \Sph_1\to [0,+\infty]$ defined on the unit circle. Except for the case of isotropic friction, the choice of a law is far from being trivial. Indeed, mechanisms producing anisotropy, as scales, hairs or textures, generally present a complex behaviour, allowing several possible 
	laws for the directional dependence of the frictional coefficient $\mu$.	
	\item \emph{Orientation is part of the shape.} For rectilinear crawlers we defined the state $x\colon \Omega \to \R$ as the function assigning to each relevant point $\xi\in\Omega$ the respective displacement. This approach is still valid in higher dimension, but may create additional issues. For instance,  in this way the information on the \lq\lq orientation\rq\rq\ of the crawler is comprised in the shape $z$. This implies that the anisotropy in friction is in the environment and not in the crawler; in other words this formulation is appropriate for locomotors with a isotropic body surface moving on an anisotropic environment, but, in general, no longer for crawler with hair or scales on their ``skin''. Such situation can be studied by introducing a state-dependence of the dissipation potential.
\end{itemize}

These elements make the class of planar crawlers very large and variegated. Whereas the study of a general and inclusive model of planar crawler is still possible, it is convenient to first reduce the complexity of the system with suitable assumptions on admissible shape changes, elasticity and dissipation. Indeed, in accordance to the the principle of simplexity, soft planar crawlers usually have only a small number of controls, in contrast with the larger theoretical number of degree of freedom. At the same time, the structure of the crawler is aimed at emphasizing the family of shape changes directly employed by the motility strategy, reducing the effects of the other deformations.
For instance, in multi-link crawlers exploiting changes in curvature to produce an ondulatory snake-like motion, longitudinal deformations are negligible. Analogously, in discrete models, considering collinear interaction between each couple of points (i.e. links with variable length), no forces are assumed for angular deformations. 

In this sense, extending our view also to rate-dependent models, we can identify a few main clusters of planar locomotion strategies.

As in rectilinear crawlers, a first class of models is given by the motility of a system composed of point masses.
The motility at quasi-static regime of a three-points planar crawler, subject to isotropic dry friction, has been studied in \cite{BorFigChe14}. In \cite{EldJac16} a four-points three-dimensional crawler moving on a plane is analysed, assuming viscous friction and including inertial terms. 
In both cases the structure of the crawler is given a family of springs joining each couple of points, with the actuation corresponding to changes in their rest lengths. 

A second class of models studied in literature is that of multi-link crawlers. In this case, the locomotor consists of a chain of $N$ joined rigid segments, and actuation is obtained by controlling the angle between each couple of adjacent segments. The motility of such systems have been studied both at a quasistatic regime, or alternating fast and slow phases, always in the case of isotropic friction (cf.~\cite{Che05} and references quoted therein).

Models with friction anisotropy become essential if we consider the  ondulatory locomotion of snakes.
In snakes, the scaled skin produces a very strong friction force opposing lateral displacements, and a weaker anisotropic friction force opposing longitudinal movements and  facilitating forward sliding.
A common way to formulate snake-like locomotion is to adopt \lq\lq snake in a tube\rq\rq\  models: lateral slipping is prohibited (friction is infinite in every direction that is not longitudinal), resulting in a gliding of the snake along a sinuous \lq\lq track\rq\rq\ (cf.~\cite{CicDeS15,Hu2009} and references quoted therein). In these situations, motility is no longer produced by contraction/elongation cycles, by changes in the spontaneous curvature.

These examples illustrate, once again, how planar crawling is a vast and challenging topic. 
We believe that the framework of rate-independent systems can be useful also in the study of such situations. Yet, as we proposed to show with our brief digression, additional -- but commonplace -- assumptions on the mechanical properties of the crawler may be mandatory in order to reduce the complexity of the system, and thus allow a deep and detailed analytical description of the motility. 

We now return our focus to unidimensional rectilinear crawlers, introducing a different source of complexity: friction manipulation.




\section{Crawling with time-dependent friction}
\label{sec:dep_theory}

As we discussed in the introduction, some crawlers have the ability to change the friction coefficients in time, improving their locomotion capabilities. To incorporate this feature in our description, we have to introduce a time-dependence of the dissipation potential $\RR$. 

Thus, analogously to what we have done in Section \ref{sec:sub_indep}, we study the  evolution of a system  described by the force balance
\begin{equation}
0\in \partial_{\dot x} \RR(t,\dot x(t))+D_x\EE(t,x)
\label{eq:EvP_time}
\end{equation}
We say that a continuous function $x\colon [0,T]\to X$, differentiable for almost all $t\in(0,T)$, is a \emph{solution} of the problem \eqref{eq:EvP_time} with starting point $x_0$ if $x(0)=x_0$ and  $x(t)$  satisfies \eqref{eq:EvP_time} for almost all $t\in(0,T)$. As for the case of time-independent friction, an initial condition $x(0)=x_0$ is admissible if
\begin{equation}
-D_x\EE(0,x_0)\in  \partial \RR(0,0)=\C (0)
\label{eq:RISadmisgen_time}
\end{equation} 

To extend our approach to the case of time-dependent friction, we have to assume on $\RR$ some conditions of regularity in time, in order to avoid criticalities analogous to those discussed in \cite{Alessi16} for state-dependent friction. Moreover, to simplify our exposition, we consider dissipation potentials with values in $[0,+\infty)$, instead of $[0,+\infty]$.

\begin{defin} \label{def:Psi_regular}
	Let $\Psi\colon X\to [0,+\infty)$ be convex, lower semi-continuous and positive homogeneous of degree one. A function $\RR\colon [0,T]\times X\to [0,+\infty)$ is a \emph{$\Psi$-regular dissipation potential} if the following conditions are satisfied:
	\begin{enumerate}[label=\textup{($\Psi$\arabic*)},start=0]
		\item for every $t\in [0,T]$, the functional $\RR(t,\cdot)\colon X\to [0,+\infty)$ is convex, continuous and positive homogeneous of degree one;  \label{cond:regular0}
		\item there exist two positive constants $\alpha^* \geq \alpha_*>0$ such that \label{cond:regular1}
		$$
		\alpha_* \Psi (x)\leq \RR(t,x) \leq \alpha^*\Psi (x)
		\qquad \text{for every $(t,x)\in[0,T]\times X$}
		$$	
		\item there exists a constant $\Lambda>0$ such that \label{cond:regular2}
		$$
		\abs{\RR(t,x)-\RR(s,x)}\leq \Lambda\abs{t-s}\Psi (x) \qquad \text{for every $t,s\in [0,T]$, $x\in X$}
		$$	
	\end{enumerate}
\end{defin}

We now replace Theorem \ref{th:mielke} with the following result by Heida and Mielke \cite[Theorems 4.1 and 4.7]{HeiMie16}.

\begin{theorem}\label{th:heida_mielke}
	Let us consider the evolution of a system \eqref{eq:EvP_time}, where $\EE$ is of the form $\EE(t,x)=\scal{\A x}{x}-\scal{\ell(t)}{x}$, with $\A\colon X \to X^*$ symmetric positive-definite,  $\ell(t)\in W^{1,\infty}([0,T],X^*)$, and $\RR\colon [0,T]\times X\to [0,+\infty)$  a $\Psi$-regular dissipation potential. Moreover, we assume that there exists a constant $c_\Psi>0$ such that $\Psi(x)\leq c_\Psi\norm{x}_X$ for every $x\in X$. 
	Then, for every initial point $x_0$ satisfying \eqref{eq:RISadmisgen_time} there exists a unique solution $x\in\Clip([0,T],X)$, with $x(0)=x_0$, that satisfies \eqref{eq:EvP_time} for almost all $t\in [0,T]$.
\end{theorem}

We notice that the Theorem \ref{th:heida_mielke} considers only the case of a positive-definite matrix $\A$, and therefore cannot be applied directly to our case, characterized by invariance for certain translations. To be able to deal with such general situation, and recover existence and uniqueness of a solution for \eqref{eq:EvP_time}, we consider the following additional assumption on the time-dependent dissipation. \emph{
	\begin{enumerate}[label=\textup{($\diamond$)}]
		\item The functional $\RR(t,\cdot)\colon X\to \R$ satisfies \ref{cond:symmbreak} for almost every $t\in[0,T]$.
		\label{cond:symmbreak_time}
	\end{enumerate}
	\begin{enumerate}[label=\textup{($\diamond\diamond$)}]
		\item There exists a functional $\Psi\colon X\to [0,+\infty)$ convex, lower semi-continuous, positively homogeneous of degree one,  such that
		$\RR$ is $\Psi$-regular \label{cond:symmbreak_coerc} and
		\begin{itemize}
			\item  there exist two positive constants $c_1,c_2$ such that 
			$$ 
			\Psi\bigl(\chi(z,v) \bigr)\geq c_1 \norm{v} -c_2 \norm{z}_{Z}
			$$
			\item there exists a constant $c_\Psi>0$ such that $\Psi(x)\leq c_\Psi\norm{x}_X$ for every $x\in X$.
		\end{itemize}
\end{enumerate}}

	We notice that, by \ref{cond:symmbreak_coerc}, for every $z\in Z$ the restriction  $\Psi\bigl(\chi(z,\cdot) \bigr)\colon Y\to [0,+\infty)$ is coercive. Moreover, as a consequence of \ref{cond:regular1}, this fact implies also that the restrictions $\RR\bigl(t,\chi(w,\cdot) \bigr)\colon Y\to [0,+\infty)$ are coercive for every choice of $w\in Z$. Hence, it is possible to define a function $\vm\colon [0,T]\times Z\to Y$ such that
\begin{equation*}
\RR(t,\chi(w,v))-\RR(t,\chi(w,\vm(t,w)))\geq 0 \qquad \text{for every $(t,w,v)\in [0,T]\times Z\times Y$}
\end{equation*}
Furthermore, $\vm(t,\cdot)\colon Z\to Y$ is univocally determined for almost every $t\in[0,T]$.

We also observe that \ref{cond:regular1} and \ref{cond:symmbreak_coerc} determine an uniform bound on $\vm(t,w)$, depending only on $\norm{w}_Z$. 
Let us recall that, since $\chi$ is a linear and continuous operator, there exists a constant $c_\chi>0$ such that $\norm{\chi(w,0)}_X\leq  c_\chi \norm{w}_Z$. Thus we have the estimate
\begin{equation}\label{eq:shape_bound}
\RR(t,\chi(w,\vm(t,w)))\leq \RR(t,\chi(w,0))\leq \alpha^*c_\Psi\norm{\chi(w,0)} \leq \alpha^*c_\Psi c_\chi \norm{w}_Z
\end{equation}

Recalling that we are assuming on the energy $\EE(t,x)=\Esh(t,\sigma(x))$ the same properties discussed in Section \ref{sec:indep_theory}, we are now ready to state our main result.

\begin{theorem}\label{th:dep}
	In the framework above,	given $\ell \in W^{1,\infty}([0,T],X^*)$ and an admissible initial point $x_0$, assuming that  conditions \ref{cond:symmbreak_time} and \ref{cond:symmbreak_coerc} are satisfied, then there exists a unique solution $x\in\Clip([0,T],X)$, of the evolution problem \eqref{eq:EvP_time} that satisfies the initial condition $x(0)=x_0$.
\end{theorem}

\begin{proof} 
	
	We follow the same line of the proof of Theorem \ref{th:indep}. The evolution \eqref{eq:EvP_time} of our system is solution of the variational inequality
	\begin{equation} \label{eq:VI_dep}
	\scal{\A x(t)-\ell (t)}{u-\dot x(t)}+\RR(t,u)-\RR(t,\dot x(t))\geq 0 
	\end{equation}
	for every $u\in X$ and almost every $t\in[0,T]$.
	We consider the class of functions $u\in X$ such that $u=\chi(\sigma(\dot x(t)),v)$, so that, for almost every $t\in[0,T]$, we have
	\begin{equation}
	\RR\left(t,\chi\bigl(\sigma(\dot x(t)),v\bigr) \right)-\RR(t,\dot x(t))\geq 0 \qquad \text{for every $v\in Y$}
	\end{equation}
	As in the time-independent case, by condition \ref{cond:symmbreak_time} we know that, for every solution $x(t)=\chi(z(t),y(t))$ of \eqref{eq:VI_dep}, and for almost every time, we can express the derivative of the net translation $\dot y(t)$ as a function of that of the shape change $\dot z(t)$, namely
	\begin{equation} 
	\dot y(t)=\vm (t,\dot z(t)):=\argmin_{v\in Y}\ \RR(t,\chi(\dot z(t),v))
	\label{eq:vm_time}
	\end{equation}
We remark that, whereas \ref{cond:symmbreak_time} assures the existence of a \emph{unique} minimizer of $\RR\bigl(t,\chi(w,\cdot))\bigr)$  for \emph{almost every} time $t\in [0,T]$, the existence of \emph{at least one} minimizer for \emph{every} $t\in [0,T]$ is obtained by a combination of \ref{cond:regular2} and \ref{cond:symmbreak_coerc}, implying the coercivity of the restrictions $\RR\bigl(t,\chi(w,\cdot))\bigr)$. Therefore the function $\vm\colon[0,T]\times Z\to Y$ is univocally defined for almost every $t$. Moreover, it can be assumed to be positively homogeneous of degree one in~$z$.

 Thus it is sufficient to study the evolution of the shape $z(t)$, which we know to be a solution of the shape-restricted variational equation
	\begin{equation}
	\scal{\Ash z(t)-\lsh (t)}{w-\dot z(t)}+\Dsh (t,w)-\Dsh (t,\dot z(t))\geq 0 \quad \text{for every $w\in Z$} 
	\label{eq:RVI_time}
	\end{equation}
	where the  shape-restricted dissipation $\Dsh\colon[0,T]\times Z\to[0,+\infty) $ is defined as 
	\begin{equation}\label{xeq:Dsh_time}
	\Dsh (t,w):=\inf_{y\in Y} \RR(t,\chi(w,y))=\RR\bigl(t,\chi(w,\vm(w))\bigr)
	\end{equation}
	We observe that each  function $z\colon [0,T]\to Z$ is a solution of \eqref{eq:RVI_time} if and only if it satisfies 
	\begin{equation} \label{eq:SF_time}
	0\in \partial \Dsh (\dot z(t))+D_z\Esh(t,z(t)) 
	\end{equation}
	 for almost every $t\in[0,T]$.
	
	Our plan is to apply Theorem \ref{th:heida_mielke} to \eqref{eq:SF_time} and then recover the displacement $y(t)$ of the crawler through the relationship \eqref{eq:vm_time}. To do so, first of all we observe that the desired properties of the dissipation potential are preserved by the dimensional reduction:
	\begin{lemma} \label{lem:propDsh_time}
		The time-dependent shape-restricted dissipation functional $\Dsh$ defined in \eqref{xeq:Dsh_time} is $\Psh$-regular dissipation potential for a suitable $\Psh\colon Z\to [0,+\infty)$ , such that $\Psh(z)\leq c' \norm{z}_Z$ for a constant $c'>0$.
	\end{lemma} 
	We prove the lemma at the end of this proof. The admissibility condition for the initial point becomes, for the shape-restricted problem,
	\begin{equation} \label{eq:init_admiss_dep}
	-D_z\Esh(0,z_0)\in \partial\Dsh(0,0):=\Csh (0)
	\end{equation}
	Analogously to the time-independent case, \eqref{eq:init_admiss_dep} is automatically satisfied when \eqref{eq:RISadmisgen_time} is satisfied. Hence, Theorem \ref{th:heida_mielke} assures us the existence of an unique solution $z(t)$ of the evolution problem \eqref{eq:SF_time}.

	As a consequence, every solution $\bar x(t)=\chi (\bar z, \bar y)$ of \eqref{eq:EvP_time} is such that  $\bar z$ is the unique solution of \eqref{eq:SF_time}.
	Moreover, we know that if a solution exists, its net displacement $\bar y(t)$ is uniquely defined as
	\begin{equation}
	\bar y(t):=y_0 + \int_{0}^{t} \vm(t,\dot{\bar{z}}(s)) \dd s
	\end{equation}
	where $y_0=\pi (x_0)$.  This provides the uniqueness of the solution of \eqref{eq:EvP_time}. Yet, we still have to prove the existence of such a solution, since in this case we do not have already available a result for a positive-semidefinite energy. To do so, we show that the functional
	\begin{equation}
	G[v]:=\int_{0}^{T} g(t,v(t))\dd t \qquad \text{with $g(t,y)=\RR\left(t,\chi\left(\dot{\bar z}(t), y\right)\right)$}
	\end{equation}
	has a minimum point in $L^1([0,T],\R^N)$. We observe that $g(t,y)$ is always non-negative, measurable in $t$, continuous and convex in $y$. Thus the functional $G$ is weakly sequentially lower semicontinuous in $L^p([0,T],\R^N)$, for every $p\geq 1$ (cf.~for instance \cite[Th.~3.20]{dacorogna}). 
	
	Let us now recall that, by Theorem \ref{th:heida_mielke} the solution $\bar z(t)$ of \eqref{eq:SF_time} is Lipschitz continuous, and therefore there exists a constant $M>0$ such that $\norm{\dot{\bar z}(t)}_Z<M$ for almost every $t\in[0,T]$. Thus, by \eqref{eq:shape_bound}, we know that 
	\begin{equation}
	\RR(t,\chi(\dot{\bar z}(t),0))\leq  \alpha^*c_\Psi c_\chi M=:M_1
	\end{equation}
	On the other hand, combining \ref{cond:regular1} and \ref{cond:symmbreak_coerc} we get the estimate
	\begin{equation}
	\RR(t,\chi(\dot{\bar z}(t),v))\geq  \alpha_*\Psi(\chi(\dot{\bar z}(t),v))
	\geq \alpha_*c_1\norm{v}-\alpha_*c_2M
	\end{equation}
	Hence
	\begin{equation}
	\RR(t,\chi(\dot{\bar z}(t),v))\geq  M_1 \qquad \text{for $\abs v\geq M_2:=\frac{M_1+\alpha_*c_2M}{\alpha_*c_1}$}
	\end{equation}
	
	We now consider a minimizing sequence $v_j\in L^1([0,T],\R^N)$ for $G[v]$, meaning that
	$$
	\lim_{j\to \infty} G[v_j]=\inf_{v\in L^1([0,T],\R^N)}G[v] 
	$$
	We construct a second sequence $\hat v_j\in L^1([0,T],\R^N)\cap L^\infty([0,T],\R^N)$ defined as
	\begin{equation}
	\hat{v}_j(t)=\begin{cases}
	v_j(t) &\text{if $\abs{v_j(t)}\leq M_2$}\\
	0 &\text{if $\abs{v_j(t)}>M_2$}\\
	\end{cases}
	\end{equation}
	We observe that $\norm{\hat v_j}_{L^\infty([0,T],\R^N)}\leq M_2$ and $G[\hat v_j]\leq G[v_j]$. Hence also $\hat v_j$ is a minimizing sequence for $G$.
	
	Moreover, we notice that $\norm{\hat v_j}_{L^2([0,T],\R^N)}\leq TM_2^2$. Since $L^2$ is reflexive, we deduce the existence of a subsequence $\hat v_{j_k}$ converging to $\bar v\in L^2([0,T],\R^N)\subset  L^1([0,T],\R^N)$ in the weak topology of $L^2$. Moreover, $\norm{\bar v}_{L^\infty([0,T],\R^N)}\leq M_2$.
	
	On the other hand, we have shown that $G$ is  weakly sequentially lower semicontinuous in $L^2([0,T],\R^N)$, thus 
	$$
	0\leq G[\bar v] \leq \lim_{j\to \infty} G[\hat v_{j_k}]=\inf_{v\in L^1([0,T],\R^N)}G[v] 
	$$
	Hence $\bar v$ is a minimum point for $G$ and, therefore, the net displacement
	\begin{equation}
	\bar y(t):=y_0 + \int_{0}^{t} \bar v(s) \dd s
	\end{equation}
	is well defined. We have thus obtained the existence of a (unique) solution $\chi(\bar x,\bar y)$ of \eqref{eq:EvP_time}. 
	
	To conclude, we notice that $\bar y$ is Lipschitz continuous with Lipschitz constant $M_2$. Since $\bar z$ is also Lipschitz continuous and $\chi$ is a linear, continuous operator,  we obtain that also the unique solution $\bar x$ of \eqref{eq:EvP_time} is Lipschitz continuous.
\end{proof}

We observe that the characterization of the evolution problem through stasis domains still holds, as well as the properties of Lemmas \ref{lemma:Csh_immersion} and \ref{lemma:SB_normalcone}, and the relationship with the locomotion capability discussed in Remarks \ref{rem:SD} and \ref{rem:SDdirection}. Yet, now the stasis domain changes in time, and with it the locomotion properties of the crawler.
Similarly, the sweeping process interpretation of Section~\ref{sec:sweeping} remains valid, but now the domain of the process will also change shape, in addition to the translations already discussed in the case of time-independent friction.

\begin{proof}[Proof of Lemma \ref{lem:propDsh_time}]
	We show that $\Dsh\colon[0,T]\times Z\to [0,+\infty)$ satisfies the three conditions of Definition \ref{def:Psi_regular}.	
	Condition \ref{cond:regular0} follows from Lemma \ref{lem:propDsh},
	applied to each $\RR(t,\cdot)$; the part of the proof regarding lower semi-continuity can be easily adapt to apply also to the preservation of continuity.  
	To satisfy the other two conditions, let us first introduce the functional $\Psh\colon Z\to [0,+\infty)$ as
	\begin{equation}
	\Psh(w)	=\inf_{v\in Y} \Psi (\chi(w,v))
	\end{equation}
	By \ref{cond:symmbreak_coerc}, we can apply Lemma \ref{lem:propDsh} to $\Psi$ and obtain that $\Psh$ is convex, lower semi-continuous and positive homogeneous of degree one. Moreover, there exists a function $v_\Psi\colon Z\to Y$ positively homogeneous of degree one such that
	$$
	\Psh(w)	=\Psi (\chi(w,v_\Psi(w)))
	$$ 
	Since \ref{cond:regular1} holds for $\RR$ with respect to $\Psi$, we have that, for every $(t,w)\in[0,T]\times Z$,
	\begin{gather*}
	\alpha_* \Psh (w)
	= \alpha_* \Psi \bigl(\chi (w,v_\Psi(w))\bigr)
	\leq \alpha_* \Psi \bigl(\chi (w,\vm(t,w))\bigr)
	\leq\Dsh (t,w)\\
	\Dsh (t,w)\leq \RR \bigl(t,\chi (w,v_\Psi(t,w))\bigr)
	\leq \alpha^* \Psi \bigl(\chi (w,v_\Psi(t,w))\bigr)
	=\alpha^* \Psh (w)
	\end{gather*}
	and thus \ref{cond:regular1} holds for $\Dsh$ with respect to $\Psh$ with the same constants $\alpha^*$ and $\alpha_*$ .
	
	We now show that also \ref{cond:regular2} holds for $\Dsh$ and $\Psh$, with constant $\Lambda_\mathrm{sh}(\tau)=\frac{\alpha^*}{\alpha_*} \Lambda $. Let us fix any $w\in Z$, $t,s\in [0,T]$ and assume, without loss of generality, that $\Dsh(t,w)\geq \Dsh(s,w)$. This implies that 
	$$\RR\bigl(s, \chi(w,\vm(s,w)) \bigr)=\Dsh(s,w)\leq \Dsh(t,w) \leq \RR\bigl(t, \chi(w,\vm(s,w)) \bigr)$$
	By \ref{cond:regular2} for $\RR$, this implies that
	\begin{align*}
	\abs{\Dsh(t,w) -\Dsh(s,w)} &
	\leq \abs{\RR\bigl(t, \chi(w,\vm(s,w)) \bigr)-\RR\bigl(s, \chi(w,\vm(s,w)) \bigr)}\\
	&\leq \Lambda\abs{t-s}\Psi (\chi(w,\vm(s,w))
	\end{align*}
	We also observe that, by combining the estimates \ref{cond:regular1} for $\RR$ and $\Dsh$, we get
	\begin{equation*}
	\Psi (\chi(w,\vm(s,w))\leq \frac{1}{\alpha_*}\RR(s,\chi(w,\vm(s,w))
	\leq\frac{\alpha^*}{\alpha_*}\Psh (w)
	\end{equation*}
	that, combined with the inequality above, gives us
	\begin{equation}
	\abs{\Dsh(t,w) -\Dsh(s,w)}
	\leq \Lambda\abs{t-s}\frac{\alpha^*}{\alpha_*}\Psh (w)=\Lambda_\mathrm{sh}\abs{t-s}\Psh (w)
	\end{equation}
	that corresponds to \ref{cond:regular2} for $\Dsh$ with respect to the functional~$\Psh$.
	
	To conclude, let us recall that, since $\chi$ is a linear and continuous operator, there exists a constant $c_\chi>0$ such that $\norm{\chi(w,0)}_X\leq  c_\chi \norm{w}_Z$. Since by \ref{cond:symmbreak_coerc} we have
	$$
	\Psh(w)\leq \Psi(\chi(w,0))\leq  c_\Psi\norm{\chi(w,0)}_X\leq  c_\Psi c_\chi \norm{w}_Z
	$$
	that gives the desired estimate for $c'=c_\Psi c_\chi$.

\end{proof}


\section{Time-dependent friction: examples} \label{sec:dep_examples}

As we explained in the introduction, a time-dependence in the friction coefficients is an alternative way to produce a sufficient amount of anisotropy  to achieve locomotion. With respect to systems with constant anisotropic friction, time-dependent friction does not fix a priori a directionality in the crawler; its effect has to be evaluated jointly with the activation on shape.  Yet, it is usually employed in the same way as directional friction: to produce a strong resistive force against thrusts in the undesired direction, and to facilitate sliding in the other one. Indeed, the use of anchoring in stick-slip locomotion can be seen as a strong form of time-dependent friction.

Moreover, time-dependent friction considerably increases the complexity of the systems, improving the locomotion capability of the crawler. This is exemplified by a very common crawling strategy: \emph{two-anchor crawling}. Two-anchor crawlers can generally be described as a segment of variable length with time-dependent friction at its ends. Their locomotion strategy is composed of two phases, as follows. In the first phase the backward end stick to the surface and the body expands, making the other end slide forward; in the second phase it is the forward end to stick to the surface, so that the backward end is pulled forward by a contraction of the body. Examples of such locomotors are the caterpillar \cite{GriTri14}, the leech \cite{SNPK86}, and the robotic crawlers proposed in \cite{scalybot,Onal2013,Ume16c,Vikas16}.

\medbreak

In order to keep our analysis from being too burdensome, we proceed as follows. First, we give some general results about the meaning of assumptions \ref{cond:symmbreak_time} and \ref{cond:symmbreak_coerc} for discrete e continuous models for crawler. Then, we focus on a minimal discrete model with different strategies of friction manipulation, to illustrate with more detail the situation in a variation of two-anchor crawling.

\subsection{Characterization of discrete and continuous models}

Condition \ref{cond:symmbreak_time} requires that \ref{cond:symmbreak} is satisfied at almost every time; hence its meaning has already been explained by the analysis given in Section \ref{sec:indep_examples}, in particular by Lemmas \ref{lemma:SB_normalcone}, \ref{lemma:SBdiscr}, \ref{lemma:SBcont} and \ref{lemma:SBdiscr_asym}. We observe that in general \ref{cond:symmbreak} is false only for some critical choices of the friction coefficient. Thus, we can expect \ref{cond:symmbreak_time} to fail only either if the friction coefficient are constant in a time-interval assuming exactly one of those critical values, or if they present a specific \emph{ad hoc} time evolution along the submanifold of critical values.

Regarding condition \ref{cond:symmbreak_coerc}, we have the following two results, respectively for discrete and continuous models.

\begin{lemma} \label{lemma:SB2t_discr}
	Let us consider a one-dimensional discrete crawler with body $\Omega_N=\{\xi_1,\dots,\xi_N\}\subset \R$ and state $x\in X=\R^N$, where each component $x_i$ corresponds to the displacement of $\xi_i$, and such that each point $\xi$ is subject to the dry friction force:
	\begin{equation*}
	F_i(t)=F(\dot{x}_i(t)) \qquad \text{ where }\; F(v) \in
	\begin{cases}
	\displaystyle \{\mu^i_-(t)\}   &  \text{if $v<0$} \\
	\displaystyle [-\mu^i_+(t),\mu^i_-(t)]   &   \text{if $v=0$} \\
	\displaystyle \{-\mu^i_+(t)\}   &  \text{if $v>0$}
	\end{cases}
	\end{equation*}
	Let us assume that the functions $\mu^i_\pm(t)\colon [0,T]\to \R$, for $i=1,\dots, N$, are Lipschitz continuous with Lipschitz constant $L$ and that there exists two positive constants $a_1,a_2$ such that $a_1\leq \mu^i_\pm(t)\leq a_2$ for every $i=1,\dots, N$  and $t\in[0,T]$.
	Then condition \ref{cond:symmbreak_coerc} is satisfied.
\end{lemma}
\begin{proof}
	The dissipation functional $\RR\colon [0,T]\times \R^N\to \R$ is defined as
	\begin{equation*}
	\RR(t,u)=\sum_{i=1}^N \RR_i (t,u_i) \qquad\text{with}\, \RR_i(t,u_i)=\begin{cases}
	-\mu_-^i(t)\,u_i  &  \text{if $u_i\leq 0$} \\
	\mu_+^i(t)\,u_i  &  \text{if $u_i\geq 0$}
	\end{cases}
	\end{equation*}
	and we observe that, by construction, at each time $t\in [0,T]$  the functional $\RR(t,\cdot)$ is convex, continuous and positive homogeneous of degree one.
	Let us now set
	$$
	\Psi(x)=\norm{x}_1  \qquad \text{where $\norm{x}_1=\sum_{i=1}^N \abs{x_i} $}
	$$
	We have $a_1 \Psi(x)\leq \RR(t,x)\leq a_2 \Psi(x)$ for every $t\in [0,T]$. Moreover, for every $t,s\in [0,T]$ and $x\in X$, we have
	$$
	\abs{\RR(t,x)-\RR(s,x)}\leq L\abs{t-s}\Psi (x) 
	$$
	and thus $\RR$ is a $\Psi$-regular potential.
	Moreover the functional $\Psi$ is coercive, so also its restrictions to the affine subspaces $\chi(z,\R)$, with $z\in Z$, are coercive. 
To prove the required estimate, let us notice that, for every $i=1,\dots,N$, we have
$$
\abs{x_i}\geq \abs{y}- \frac{N-1}{2}\norm{z}_\infty\geq \abs{y}- \frac{N-1}{2}\norm{z}_2
$$
Hence 
$$
\Psi(\chi(z,v))=\norm{x}_1\geq N\abs{y}- \frac{N(N-1)}{2}\norm{z}_2
$$
	Finally, recalling that we are using the Euclidean norm as norm  on $X$, we have the well known inequality $\Psi(x)=\norm{x}_1\leq \sqrt{N}\norm{x}_2$.
\end{proof}

\begin{lemma}\label{lemma:SB2t_cont}
	Let us consider a continuous crawler with bounded body $\Omega=[\xi_a,\xi_b]\subset \R$ and state $x\in X=W^{1,2}(\Omega,\R)$, where $x(\xi)$ is the displacement of $\xi$, and such that each point $\xi$ is subject to the dry friction force $f_{t,\xi}$ per unit reference length:
	\begin{equation*}
	f_{t,\xi}=f(t,\xi,\dot{x}(t,\xi))\qquad \text{ where }f(t,\xi,v)\in
	\begin{cases}
	\displaystyle \{\mud_{-}(t,\xi)\}  & \text{ if $v<0$ }  \\
	\displaystyle [-\mud_{+}(t,\xi),\mud_{-}(t,\xi)]  &  \text{ if $v=0$}  \\
	\displaystyle \{-\mud_{+}(t,\xi)\}  & \text{ if $v>0$ } 
	\end{cases}
	\end{equation*}
	Let us assume that there exists two positive constants $a_1,a_2$ such that $a_1\leq \mud_\pm(t,\xi)\leq a_2$ for every $i=1,\dots, N$  and $(t,\xi)\in[0,T]\times \Omega $.
	Moreover, we assume that $\mud_\pm(t,\cdot)\in L^2(\Omega,\R)$ for every $t\in[0,T]$, and that the following Lipschitz inequalities hold
	$$
	\norm{\mud_\pm(t,\cdot)-\mud_\pm(s,\cdot)}_{L^\infty(\Omega,\R)}\leq L\abs{t-s}
	$$	
	Then condition \ref{cond:symmbreak_coerc} is satisfied.
\end{lemma}
We observe that the lemma is satisfied, for instance, by travelling waves
$$
\mud_\pm(t,\xi)=\mu_0(\xi-ct)
$$
where $\mu_0\colon \R\to (0,+\infty)$ is a positive Lipschitz continuous function and $c$ is a constant.

\begin{proof}
	The dissipation functional $\RR\colon [0,T]\times W^{1,2}(\Omega,\R)\to \R$ is defined as
	\begin{equation*} 
	\RR(t,u)=\int_{\Omega}\left[ -\mud_-(t,\xi)  u^-(\xi)+\mud_+(t,\xi)  u^+(\xi) \right]\dd \xi
	\end{equation*}
	that satisfies \ref{cond:regular0}, as shown in Section \ref{sec:hybrid}. Let us set the functional $\Psi\colon W^{1,2}(\Omega,\R)\to \R$ as
	$$
	\Psi(x)=\norm{x}_{L^1(\Omega,\R)}
	$$
	Since $\Omega$ has finite measure, precisely $\meas(\Omega)=\abs{\xi_b-\xi_a}$, the functional $\Psi$ satisfies
	$$
	\Psi(x)\leq \abs{\xi_b-\xi_a}^{\frac{1}{2}} \norm{x}_{L^{2}(\Omega,\R)}\leq \abs{\xi_b-\xi_a}^{\frac{1}{2}} \norm{x}_{W^{1,2}(\Omega,\R)}
	$$
	From this we  deduce that $\Psi$ is continuous. Since $\Psi$
	defines an alternative norm on $X$, it is also convex and positively homogeneous of degree one. 
	We have $a_1 \Psi(u)\leq \RR(t,u)\leq a_2 \Psi(u)$ for every $t\in [0,T]$. Moreover, for every $t,s\in [0,T]$ and $u\in X$, we have
	\begin{align*}
	\abs{\RR(t,u)-\RR(s,u)}&
	\leq \int_{\Omega} \left[\abs{\mud_-(t,\xi)-\mud_-(s,\xi)}  u^-(\xi)+\abs{\mud_+(t,\xi)-\mud_+(t,\xi)}  u^+(\xi) \right] \dd \xi \\
	&\leq \norm{\mud_-(t,\cdot)-\mud_-(s,\cdot)}_{L^\infty(\Omega,\R)}\norm{u^-}_{L^1(\Omega,\R)} + \\
	&\qquad \qquad +
	\norm{\mud_+(t,\cdot)-\mud_+(s,\cdot)}_{L^\infty(\Omega,\R)}\norm{u^+}_{L^1(\Omega,\R)}\\
	&\leq L\abs{t-s}\left[\norm{u^-}_{L^1(\Omega,\R)}+\norm{u^+}_{L^1(\Omega,\R)}\right]=L\abs{t-s}\Psi(u)
	\end{align*}
	and thus $\RR$ is a $\Psi$-regular potential. 
	
	To check the coercivity of the restrictions of $\Psi$ on the affine subspaces $\chi(z,\R)$, we first observe that, compared to the discrete case, $\Psi$ is no longer coercive, since the topology of the $W^{1,2}$-norm we have on $X$ is stronger that the one induced by  $\Psi$.
	Yet, the desired coercivity on the restrictions still holds, since, for $z\in Z=W_0^{1,2}(\Omega,\R)$ and $v\in Y=R$ we have
	$$
	\Psi(\chi(z,v))=\norm{z+v}_{L^1(\Omega,\R)}\geq \norm{v}_{L^1(\Omega,\R)}-\norm{z}_{L^1(\Omega,\R)}\geq \abs{\xi_b-\xi_a}\abs{v}-\abs{\xi_b-\xi_a}^{\frac{1}{2}} \norm{z}_{W^{1,2}(\Omega,\R)}
	$$
	We have therefore verified that \ref{cond:symmbreak_coerc} holds.
\end{proof}

\subsection{Motility in a one-segment discrete crawler}
\label{sec:inching}

To illustrate the potentialities and effects of introducing time-dependence in friction, we analyse now how this change influences the motility of the crawler of Example \ref{ex:2legs}, whose body is composed by only two point elements, joined by an active elastic link. Such model can be considered as a general two-anchor locomotor, for which we consider different choices of friction control.

To simplify the model, and to show that time-dependence in friction can be used alone to break the symmetry of the system , we assume that in each of the elements the friction is isotropic, with a coefficient $\mu_i(t)$ changing in time.

Therefore, for $(t,u)\in [0,T]\times \R^2$, the dissipation potential is
\begin{equation}
\RR(t,u)=\mu_1(t)\abs{u_1} +\mu_2(t)\abs{u_2}
\end{equation}
while the internal energy remains is the same of Example \ref{ex:2legs}, namely $\EE(t,x)=\frac{k}{2}(x_2-x_1-L(t))^2$.
We recall that we are adopting as shape coordinate $z=x_2-x_1$, and as net translation $y=\frac{x_1+x_2}{2}$.
We know that \ref{cond:symmbreak} is satisfied whenever $\mu_1(t)\neq \mu_2(t)$. Assuming that $\mu_1(t)$ and $\mu_2(t)$ are positive and Lipschitz continuous, thus satisfying Lemma \ref{lemma:SB2t_cont}, we can state our evolution problem in its shape-restricted form. The shape-restricted dissipation potential is
\begin{equation}
\Dsh (t,w)=\mum(t) \abs{w}\qquad \text{with $\mum(t)=\min\{\mu_1(t),\mu_2(t)\}$ }
\end{equation}
and, when \ref{cond:symmbreak} is satisfied, we can recover its associated minimizers $\vm(t,w)\in Y$ as
\begin{equation}
\vm(t,w)=\begin{cases} \displaystyle\frac{w}{2} &\text{if $\mu_1(t)>\mu_2(t)$}\\[4mm]
-\displaystyle\frac{w}{2} &\text{if $\mu_1(t)< \mu_2(t)$}\\
\end{cases}
\end{equation}
We observe that the stasis domains  change in time, namely
\begin{gather}
\C(t)=\partial \RR (t,0)=[-\mu_1(t),\mu_1(t)]\times [-\mu_2(t),\mu_2(t)]\subset X^* \\
\Csh (t)=\partial \Dsh (t,0)= [-\mum(t),\mum(t)] \subset Z^*
\end{gather}
The shape restricted evolution problem is therefore
\begin{equation}
\ten(t):=-k (z(t)-L(t))\in \partial\Dsh (\dot z(t))=\begin{cases} \{-\mum(t)\} &\text{if $\dot{z}(t)<0$}\\
[-\mum(t),\mum(t)] &\text{if $\dot{z}(t)=0$}\\
\{\mum(t)\} &\text{if $\dot{z}(t)>0$}
\end{cases} 
\end{equation}
From this we deduce that an elongation ($\dot z(t)>0$) is possible only if $k\dot L(t)-\dot\mu_\mathrm{min}(t)>0$ and $\ten(t)=\mum(t)$. The last condition means that the crawler is in a state of maximum compression; the first condition can be satisfied by an increment of the rest length, or by a reduction of the lower of the friction coefficients, or by both. We notice that we can still have an elongation even if we are reducing the rest length, provided that we reduce the friction coefficients enough. Similar conclusions can be obtained with respect to contractions ($\dot z(t)<0$).

It is not difficult to produce motility strategies, i.e.\ periodic changes in load and friction, producing a non-zero net displacement in each cycle. An easy approach could be to act at separate times on load and friction, in such a way that the final evolution is equivalent to a collage of time intervals where the systems evolves with a time-independent friction. Yet, it is not uncommon that changes in shape and friction are generated by the same mechanism; for instance, longitudinal muscular contractions in the earthworm stiffen the set\ae\ and  shorten the body segment. Since such situations are also those in which a theory with time-dependent friction is actually needed, we present here three illustrative examples, showing how control on friction can largely increase the locomotion capabilities of a crawler.

All our examples share the same input on shape, associated to a periodic elongation and contraction between $0$ and $\Lmax>0$:
\begin{equation}\label{eq:shape_act}
L(t)=\begin{cases} 2t\Lmax  &\text{for $0\leq t\leq \frac{1}{2}$}\\
2(1-t)\Lmax  &\text{for $\frac{1}{2}\leq t\leq 1$}
\end{cases}
\end{equation}

\begin{figure}[t]
	\begin{center}
		\begin{tikzpicture}[line cap=round,line join=round,>=stealth,x=0.5cm,y=0.4cm]
		\clip(-10.,1.) rectangle (21.,31.);\footnotesize
		\draw [->] (0.,24.) -- (18.,24.)node[right]{$t$};
		\draw [line width=1.pt,orange] (0.,25.)-- (8.,29.)-- (16.,25.) node[above] {$L(t)$};
		\draw [->] (0.,24.) -- (0.,30.);
		\draw [dashed] (0.,25.) node[left]{$0$} -- (18.,25.);
		\draw [dashed] (0.,29.) node[left]{$\Lmax$} -- (18.,29.);
		\draw [->] (0.,17.) -- (0.,23.);
		\draw [->] (0.,17.) -- (18.,17.)node[right]{$t$};
		\draw [->,] (0.,10.) -- (0.,16.);
		\draw [->] (0.,10.) -- (18.,10.)node[right]{$t$};
		\draw [dashed] (0.,22.) node[left]{$1.5\mu$} -- (18.,22.);
		\draw [dashed] (0.,20.) node[left]{$\mu$} -- (18.,20.);
		\draw [dashed] (18.,18.) -- (0.,18.) node[left]{$0.5\mu$};
		\draw [dashed] (0.,15.) node[left]{$1.5\mu$} -- (18.,15.);
		\draw [dashed] (18.,13.) -- (0.,13.) node[left]{$\mu$};
		\draw [dashed] (0.,11.) node[left]{$0.5\mu$} -- (18.,11.);
		\draw [->] (0.,3.) -- (0.,9.);
		\draw [->] (0.,3.) -- (18.,3.)node[right]{$t$};
		\draw [dashed] (0.,4.) node[left]{$0.5\mu$} -- (18.,4.);
		\draw [dashed] (0.,8.) node[left]{$1.5\mu$} -- (18.,8.);
		\draw [dotted] (16.,29.)-- (16.,3.) node[below]{$1$};
		\draw [dotted] (8.,29.)-- (8.,3.) node[below]{$0.5$};
		\draw [line width=1.pt,blue] (0.,20.)-- (4.,22.) -- (12.,18.)node[above] {$\mu_1(t)$}-- (16.,20.);
		\draw [line width=1.pt,cyan] (0.,20.)-- (4.,18.)-- (12.,22.)node[below] {$\mu_2(t)$}-- (16.,20.);
		\draw [line width=1.pt,cyan] (0.,15.)-- (8.,11.)-- (16.,15.)node[below right ] {$\mu_2(t)$};
		\draw [line width=1.pt,blue] (0.,11.)-- (8.,15.)-- (16.,11.)node[ above right] {$\mu_1(t)$};
		\draw [line width=1.pt,cyan] (0.,4.)-- (4.,4.)-- (8.,8.)-- (12.,8.)node[below left] {$\mu_2(t)$}-- (16.,4.);
		\draw [line width=1.pt,blue] (0.,4.)-- (4.,8.)-- (8.,8.)-- (12.,4.)node[above right] {$\mu_1(t)$}-- (16.,4.);
		\normalsize
		\draw (-1,27) node[left,text width=4cm]{Shape-change actuation strategy};
		\draw (-1,20) node[left,text width=4cm]{Friction-manipulation  strategy A};
		\draw (-1,13) node[left,text width=4cm]{Friction-manipulation  strategy B};
		\draw (-1,6) node[left,text width=4cm]{Friction-manipulation  strategy C};
		\end{tikzpicture}
	\end{center}
	\caption{The three friction-manipulation strategies proposed respectively in Example \ref{ex:stratA} (strategy A), Example \ref{ex:stratB} (strategy B) and Example \ref{ex:stratC} (strategy C), with respect to the common shape-actuation of \eqref{eq:shape_act}. }
	\label{fig:manip_strat}
\end{figure}
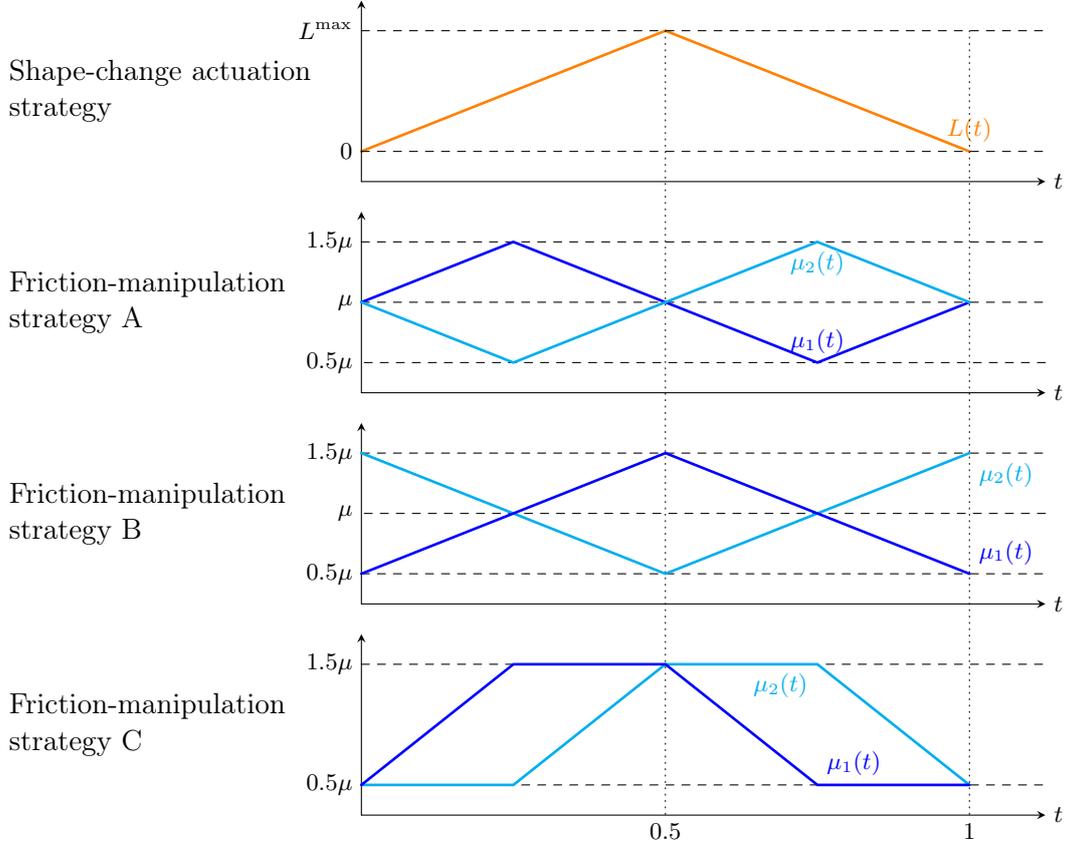

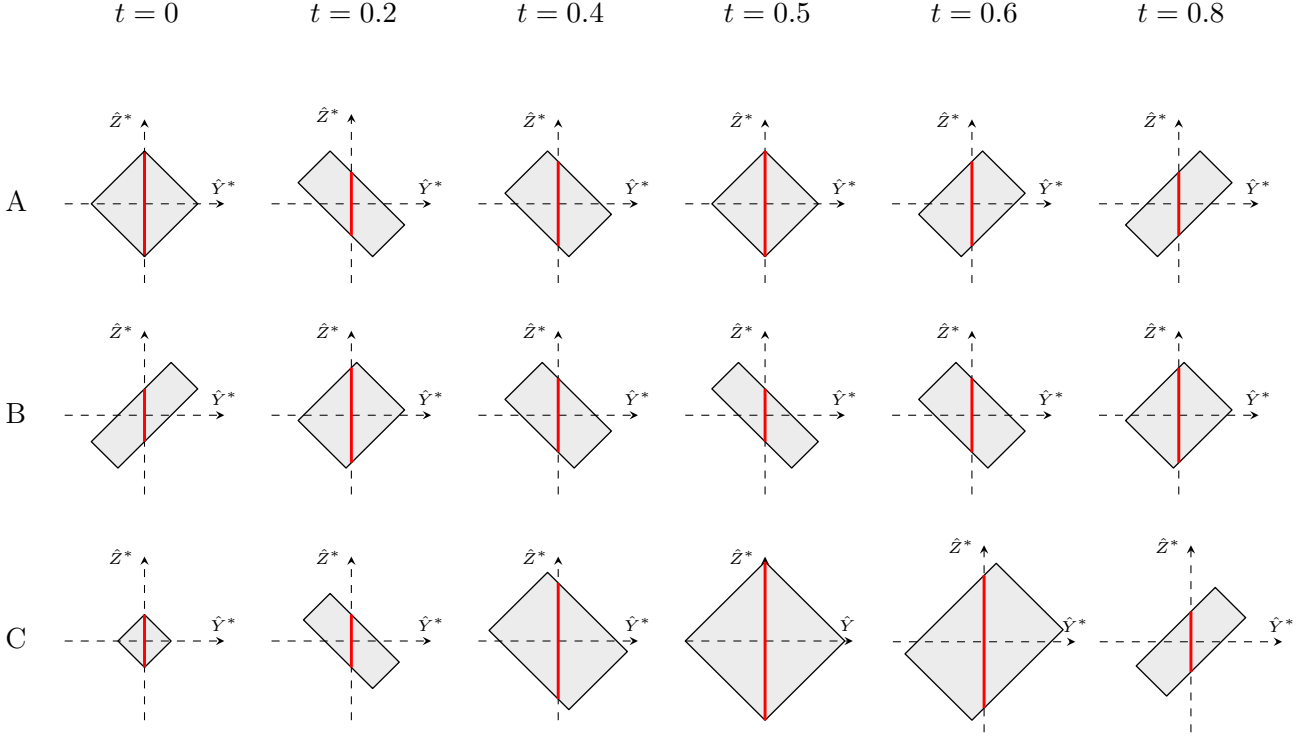
\begin{figure}[t]
	\begin{center}
			\centerline{
			\begin{tabular}{lm{23mm}m{23mm}m{23mm}m{23mm}m{23mm}m{23mm}}
				&$$t=0$$&$$t=0.2$$&$$t=0.4$$&$$t=0.5$$&$$t=0.6$$&$$t=0.8$$\\
				A&
				\begin{tikzpicture}[line join=round,>=stealth,x=0.7cm,y=0.7cm]
				\clip(-1.6,-1.8) rectangle (1.7,2);
				\draw[->,dashed, color=black] (-1.5,0.) -- (1.5,0.) node[above]{\tiny $\hat Y^*$};
				\draw[->,dashed,color=black] (0.,-1.5) -- (0.,1.6) node[left]{\tiny $\hat Z^*$};
				\draw [line width=1.1pt,color=red] (0.,1.)-- (0.,-1.);
				\fill[line width=0.5pt,draw=black,fill=uuuuuu,fill opacity=0.1] (0.,1.) -- (-1.,0.) -- (0.,-1.) -- (1.,0.) -- cycle;
				\end{tikzpicture}
				&
				\begin{tikzpicture}[line join=round,>=stealth,x=0.7cm,y=0.7cm]
				\clip(-1.6,-1.8) rectangle (1.7,2);
				\draw[->,dashed, color=black] (-1.5,0.) -- (1.5,0.) node[above]{\tiny $\hat Y^*$};
				\draw[->,dashed,color=black] (0.,-1.5) -- (0.,1.7) node[left]{\tiny $\hat Z^*$};
				\fill[line width=0.5pt,draw=black,fill=uuuuuu,fill opacity=0.1] (-0.4,1.) -- (-1.,0.4) -- (0.4,-1.) -- (1.,-0.4) -- cycle;
				\draw [line width=1.1pt,color=red] (0.,0.6)-- (0.,-0.6);
				\end{tikzpicture}
				&
				\begin{tikzpicture}[line join=round,>=stealth,x=0.7cm,y=0.7cm]
				\clip(-1.6,-1.8) rectangle (1.7,2);
				\draw[->,dashed, color=black] (-1.5,0.) -- (1.5,0.) node[above]{\tiny $\hat Y^*$};
				\draw[->,dashed,color=black] (0.,-1.5) -- (0.,1.6) node[left]{\tiny $\hat Z^*$};
				\fill[line width=0.5pt,draw=black,fill=uuuuuu,fill opacity=0.10000000149011612] (-0.2,1.) -- (-1.,0.2) -- (0.2,-1.) -- (1.,-0.2) -- cycle;
				\draw [line width=1.1pt,color=red] (0.,0.8)-- (0.,-0.8);
				\end{tikzpicture}
				&
				\begin{tikzpicture}[line join=round,>=stealth,x=0.7cm,y=0.7cm]
				\clip(-1.6,-1.8) rectangle (1.7,2);
				\draw[->,dashed, color=black] (-1.5,0.) -- (1.5,0.) node[above]{\tiny $\hat Y^*$};
				\draw[->,dashed,color=black] (0.,-1.5) -- (0.,1.6) node[left]{\tiny $\hat Z^*$};
				\fill[line width=0.5pt,draw=black,fill=uuuuuu,fill opacity=0.1] (0.,1.) -- (-1.,0.) -- (0.,-1.) -- (1.,0.) -- cycle;
				\draw [line width=1.1pt,color=red] (0.,1.)-- (0.,-1.);
				\end{tikzpicture}
				&
				\begin{tikzpicture}[line join=round,>=stealth,x=0.7cm,y=0.7cm]
				\clip(-1.6,-1.8) rectangle (1.7,2);
				\draw[->,dashed, color=black] (-1.5,0.) -- (1.5,0.) node[above]{\tiny $\hat Y^*$};
				\draw[->,dashed,color=black] (0.,-1.5) -- (0.,1.6) node[left]{\tiny $\hat Z^*$};
				\fill[line width=0.5pt,draw=black,fill=uuuuuu,fill opacity=0.1] (0.2,1.) -- (-1.,-0.2) -- (-0.2,-1.) -- (1.,0.2) -- cycle;
				\draw [line width=1.1pt,color=red] (0.,0.8)-- (0.,-0.8);
				\end{tikzpicture}&
				\begin{tikzpicture}[line join=round,>=stealth,x=0.7cm,y=0.7cm]
				\clip(-1.6,-1.8) rectangle (1.7,2);
				\draw[->,dashed, color=black] (-1.5,0.) -- (1.5,0.) node[above]{\tiny $\hat Y^*$};
				\draw[->,dashed,color=black] (0.,-1.5) -- (0.,1.6) node[left]{\tiny $\hat Z^*$};
				\fill[line width=0.5pt,draw=black,fill=uuuuuu,fill opacity=0.1] (0.4,1.) -- (-1.,-0.4) -- (-0.4,-1.) -- (1.,0.4) -- cycle;
				\draw [line width=1.1pt,color=red] (0.,0.6)-- (0.,-0.6);
				\end{tikzpicture}\\
				B
				&
				\begin{tikzpicture}[line join=round,>=stealth,x=0.7cm,y=0.7cm]
				\clip(-1.6,-1.8) rectangle (1.7,2);
				\draw[->,dashed, color=black] (-1.5,0.) -- (1.5,0.) node[above]{\tiny $\hat Y^*$};
				\draw[->,dashed,color=black] (0.,-1.5) -- (0.,1.6) node[left]{\tiny $\hat Z^*$};
				\fill[line width=0.5pt,draw=black,fill=uuuuuu,fill opacity=0.1] (0.5,1.) -- (-1.,-0.5) -- (-0.5,-1.) -- (1.,0.5) -- cycle;
				\draw [line width=1.1pt,color=red] (0.,0.5)-- (0.,-0.5);
				\end{tikzpicture}
				&
				\begin{tikzpicture}[line join=round,>=stealth,x=0.7cm,y=0.7cm]
				\clip(-1.6,-1.8) rectangle (1.7,2);
				\draw[->,dashed, color=black] (-1.5,0.) -- (1.5,0.) node[above]{\tiny $\hat Y^*$};
				\draw[->,dashed,color=black] (0.,-1.5) -- (0.,1.6) node[left]{\tiny $\hat Z^*$};
				\fill[line width=0.5pt,draw=black,fill=uuuuuu,fill opacity=0.1] (0.1,1.) -- (-1.,-0.1) -- (-0.1,-1.) -- (1.,0.1) -- cycle;
				\draw [line width=1.1pt,color=red] (0.,0.9)-- (0.,-0.9);
				\end{tikzpicture}
				&
				\begin{tikzpicture}[line join=round,>=stealth,x=0.7cm,y=0.7cm]
				\clip(-1.6,-1.8) rectangle (1.7,2);
				\draw[->,dashed, color=black] (-1.5,0.) -- (1.5,0.) node[above]{\tiny $\hat Y^*$};
				\draw[->,dashed,color=black] (0.,-1.5) -- (0.,1.6) node[left]{\tiny $\hat Z^*$};
				\fill[line width=0.5pt,draw=black,fill=uuuuuu,fill opacity=0.1] (-0.3,1.) -- (-1.,0.3) -- (0.3,-1.) -- (1.,-0.3) -- cycle;
				\draw [line width=1.1pt,color=red] (0.,0.7)-- (0.,-0.7);
				\end{tikzpicture}
				&
				\begin{tikzpicture}[line join=round,>=stealth,x=0.7cm,y=0.7cm]
				\clip(-1.6,-1.8) rectangle (1.7,2);
				\draw[->,dashed, color=black] (-1.5,0.) -- (1.5,0.) node[above]{\tiny $\hat Y^*$};
				\draw[->,dashed,color=black] (0.,-1.5) -- (0.,1.6) node[left]{\tiny $\hat Z^*$};
				\fill[line width=0.5pt,draw=black,fill=uuuuuu,fill opacity=0.1] (-0.5,1.) -- (-1.,0.5) -- (0.5,-1.) -- (1.,-0.5) -- cycle;
				\draw [line width=1.1pt,color=red] (0.,0.5)-- (0.,-0.5);
				\end{tikzpicture}
				&
				\begin{tikzpicture}[line join=round,>=stealth,x=0.7cm,y=0.7cm]
				\clip(-1.6,-1.8) rectangle (1.7,2);
				\draw[->,dashed, color=black] (-1.5,0.) -- (1.5,0.) node[above]{\tiny $\hat Y^*$};
				\draw[->,dashed,color=black] (0.,-1.5) -- (0.,1.6) node[left]{\tiny $\hat Z^*$};
				\fill[line width=0.5pt,draw=black,fill=uuuuuu,fill opacity=0.1] (-0.3,1.) -- (-1.,0.3) -- (0.3,-1.) -- (1.,-0.3) -- cycle;
				\draw [line width=1.1pt,color=red] (0.,0.7)-- (0.,-0.7);
				\end{tikzpicture}
				&
				\begin{tikzpicture}[line join=round,>=stealth,x=0.7cm,y=0.7cm]
				\clip(-1.6,-1.8) rectangle (1.7,2);
				\draw[->,dashed, color=black] (-1.5,0.) -- (1.5,0.) node[above]{\tiny $\hat Y^*$};
				\draw[->,dashed,color=black] (0.,-1.5) -- (0.,1.6) node[left]{\tiny $\hat Z^*$};
				\fill[line width=0.5pt,draw=black,fill=uuuuuu,fill opacity=0.1] (0.1,1.) -- (-1.,-0.1) -- (-0.1,-1.) -- (1.,0.1) -- cycle;
				\draw [line width=1.1pt,color=red] (0.,0.9)-- (0.,-0.9);
				\end{tikzpicture}\\
				C
				&
				\begin{tikzpicture}[line join=round,>=stealth,x=0.7cm,y=0.7cm]
			\clip(-1.6,-1.8) rectangle (1.7,2);
				\draw[->,dashed, color=black] (-1.5,0.) -- (1.5,0.) node[above]{\tiny $\hat Y^*$};
				\draw[->,dashed,color=black] (0.,-1.5) -- (0.,1.6) node[left]{\tiny $\hat Z^*$};
				\fill[line width=0.5pt,draw=black,fill=uuuuuu,fill opacity=0.1] (0.,0.5) -- (-0.5,0.) -- (0.,-0.5) -- (0.5,0.) -- cycle;
				\draw [line width=1.1pt,color=red] ((0.,0.5)-- (0.,-0.5);
				\end{tikzpicture}
				&
				\begin{tikzpicture}[line join=round,>=stealth,x=0.7cm,y=0.7cm]
	\clip(-1.6,-1.8) rectangle (1.7,2);
				\draw[->,dashed, color=black] (-1.5,0.) -- (1.5,0.) node[above]{\tiny $\hat Y^*$};
				\draw[->,dashed,color=black] (0.,-1.5) -- (0.,1.6) node[left]{\tiny $\hat Z^*$};
				\fill[line width=0.5pt,draw=black,fill=uuuuuu,fill opacity=0.1] (-0.4,0.9) -- (-0.9,0.4) -- (0.4,-0.9) -- (0.9,-0.4)  -- cycle;
				\draw [line width=1.1pt,color=red] (0.,0.5)-- (0.,-0.5);
				\end{tikzpicture}
				&
				\begin{tikzpicture}[line join=round,>=stealth,x=0.7cm,y=0.7cm]
	\clip(-1.6,-1.8) rectangle (1.7,2);
				\draw[->,dashed, color=black] (-1.5,0.) -- (1.5,0.) node[above]{\tiny $\hat Y^*$};
				\draw[->,dashed,color=black] (0.,-1.5) -- (0.,1.6) node[left]{\tiny $\hat Z^*$};
				\fill[line width=0.5pt,draw=black,fill=uuuuuu,fill opacity=0.1] (-0.2,1.3) -- (-1.3,0.2) -- (0.2,-1.3) -- (1.3,-0.2)  -- cycle;
				\draw [line width=1.1pt,color=red] (0.,1.1)-- (0.,-1.1);
				\end{tikzpicture}
				&
				\begin{tikzpicture}[line join=round,>=stealth,x=0.7cm,y=0.7cm]
					\clip(-1.6,-1.8) rectangle (1.7,2);
				\draw[->,dashed, color=black] (-1.5,0.) -- (1.65,0.) node[above]{\tiny $\hat Y^*$};
				\draw[->,dashed,color=black] (0.,-1.5) -- (0.,1.6) node[left]{\tiny $\hat Z^*$};
				\fill[line width=0.5pt,draw=black,fill=uuuuuu,fill opacity=0.1] (0.,1.5) -- (-1.5,0.) -- (0.,-1.5) -- (1.5,0.) -- cycle;
				\draw [line width=1.1pt,color=red] (0.,1.5)-- (0.,-1.5);
				\end{tikzpicture}
				&
				\begin{tikzpicture}[line join=round,>=stealth,x=0.8cm,y=0.8cm]
				\clip(-1.6,-1.8) rectangle (1.7,2);
				\draw[->,dashed, color=black] (-1.5,0.) -- (1.5,0.) node[above]{\tiny $\hat Y^*$};
				\draw[->,dashed,color=black] (0.,-1.5) -- (0.,1.6) node[left]{\tiny $\hat Z^*$};
				\fill[line width=0.5pt,draw=black,fill=uuuuuu,fill opacity=0.1] (0.2,1.3) -- (-1.3,-0.2) -- (-0.2,-1.3) -- (1.3,0.2) -- cycle;
				\draw [line width=1.1pt,color=red] (0.,1.1)-- (0.,-1.1);
				\end{tikzpicture}
				&
				\begin{tikzpicture}[line join=round,>=stealth,x=0.8cm,y=0.8cm]
				\clip(-1.6,-1.8) rectangle (1.7,2);
				\draw[->,dashed, color=black] (-1.5,0.) -- (1.5,0.) node[above]{\tiny $\hat Y^*$};
				\draw[->,dashed,color=black] (0.,-1.5) -- (0.,1.6) node[left]{\tiny $\hat Z^*$};
				\fill[line width=0.5pt,draw=black,fill=uuuuuu,fill opacity=0.1] (0.4,0.9) -- (-0.9,-0.4) -- (-0.4,-0.9) -- (0.9,0.4) -- cycle;
				\draw [line width=1.1pt,color=red] (0.,0.5)-- (0.,-0.5);
				\end{tikzpicture}
			\end{tabular}
		}
	\end{center}
\caption{The three friction-manipulation strategies of Example \ref{ex:stratA} (strategy A), Example \ref{ex:stratB} (strategy B) and Example \ref{ex:stratC} (strategy C), illustrate with the evolution of their respective domains $\C(t)$. The red sections correspond to the sections $\hCsh(t)$. By Remark \ref{rem:SDdirection} we notice how the locomotion capability of the crawler are qualitatively changed in time, since the same shape change can produce displacement in different directions.}
\label{fig:domain_strat}
\end{figure}

We propose three different strategies for the changes in friction, illustrated in Figure \ref{fig:manip_strat} and, with the evolution of their corresponding stasis domains, in Figure \ref{fig:domain_strat}.

\begin{example}[Strategy A] \label{ex:stratA}
	In the first case we assume that both the friction coefficient oscillate between $0.5\mu$ and $1.5\mu$, in such a way that $\mu_1(t)+\mu_2(t)=2\mu$ and that the times of maximum difference between the two coefficients correspond to the extreme values of $L(t)$. More precisely, we have
	
\begin{equation*}
\mu_1(t)=\begin{cases} (1+2t)\mu  &\text{for $0\leq t\leq \frac{1}{4}$}\\
(2-2t)\mu  &\text{for $\frac{1}{4}\leq t\leq \frac{3}{4}$}\\
(2t-1)\mu  &\text{for $\frac{3}{4} \leq t\leq 1$}
\end{cases}
\qquad
\mu_2(t)=\begin{cases} (1-2t)\mu  &\text{for $0\leq t\leq \frac{1}{4}$}\\
2t\mu  &\text{for $\frac{1}{4}\leq t\leq \frac{3}{4}$}\\
(3-2t)\mu  &\text{for $\frac{3}{4} \leq t\leq 1$}
\end{cases}
\end{equation*}

The evolution of the stasis domain $\Csh(t)$ in time is illustrated in Figure \ref{fig:stratA}.
We assume that at the initial time the crawler is in the state of maximum shape length $z$ compatible with  \eqref{eq:init_admiss_dep}, namely
$z(0)=\frac{\mu}{k}$  [equivalently $\ten(t)=-\mum(t)$].
We will check that after a cycle the system returns in this initial situation.

If  $k\Lmax>\mu$, in the first half period we have $\dot L(t)= 2\Lmax$ and $k\dot L(t)> \dot \mu_\mathrm{min}(t)$, so that the crawler remains initially steady, with $\ten(t)$ moving across the stasis domain $\Csh(t)$. We have three possible situations. 
If $\mu<k\Lmax<2\mu$, then $\ten(t)$ will remain inside the stasis domain for the whole first half period, and no motion will occur. 

If $k\Lmax> 3\mu$, then $\ten(t)$ will reach the other boundary  of $\Csh$ at a time $t_1<\frac{1}{4}$ and $\ten(t)=\mum(t)$ will remain satisfied for all $t\in[t_1,\frac{1}{2}]$, with $z(t)$ increasing. More precisely
\begin{equation}
	\dot z(t)=\begin{cases}
	0 &\text{for $t\in(0,t_1)$}\\
	2(\Lmax +\frac{\mu}{k}) &\text{for $t\in\left(t_1,\frac{1}{4}\right)$}\\
	2(\Lmax -\frac{\mu}{k}) &\text{for $t\in\left(\frac{1}{4},\frac{1}{2}\right)$}
	\end{cases} \qquad\text{with $t_1=\frac{\mu}{k\Lmax+\mu}$}
\end{equation}
In Figure \ref{fig:stratA} is illustrated, for $k\Lmax=4\mu$ the evolution of $\ten(t)$ and the corresponding net translation  produced.

If $2\mu\leq k\Lmax\leq 3\mu$, then $\ten(t)$ will reach the other boundary of $\Csh$ at a time $t_2>\frac{1}{4}$ and $\ten(t)=\mum(t)$ will remain satisfied for all $t\in[t_2,\frac{1}{2}]$, with $z(t)$ increasing. More precisely
\begin{equation}
\dot z(t)=\begin{cases}
0 &\text{for $t\in(0,t_2)$}\\
2(\Lmax -\frac{\mu}{k}) &\text{for $t\in\left(t_2,\frac{1}{2}\right)$}
\end{cases} \qquad\text{with $t_2=\frac{\mu}{2(k\Lmax-\mu)}$}
\end{equation}
By a symmetry argument one can check that, in the second half period $t\in\left(\frac{1}{2}, 1\right)$, the systems evolves according to
\begin{equation} \label{eq:inching_symm}
\dot x_1(t)=\dot x_2\left(t-\frac{1}{2}\right) \qquad \dot x_2(t)=\dot x_1\left(t-\frac{1}{2}\right)	
\end{equation}
and after a cycle the locomotors returns (up to a translation) in its original internal state.
The displacement produced by a cycle is
$$
y(1)-y(0)=\begin{cases}
0 & \text{if $ k\Lmax\leq 2\mu$} \\
\Lmax-2\frac{\mu}{k}>0 & \text{if $k\Lmax\geq 2\mu$}
\end{cases}
$$

The case $k\Lmax=\mu$ does not satisfy condition \ref{cond:symmbreak_time} and has therefore to be excluded by our analysis.

If instead $k\Lmax<\mu$, in the first quarter of the period we remain on the initial boundary point of $\Csh(t)$, with $\ten(t)=-\mum(t)$ and $z$ decreasing. In other words, we observe a contraction of the crawler with $\xi_2$ moving backwards. This is caused by the rapid drop in the corresponding friction coefficient, occurring at a faster scale than the change in the load of the spring, which would instead favour an elongation. This behaviour, however, is only a temporary phase. In the remaining time of this period, and in the following ones, the tension  $\ten(t)$ will remain inside the stasis domain for almost every time. The crawler will not move any more; the load applied to the system being too small.
\end{example}

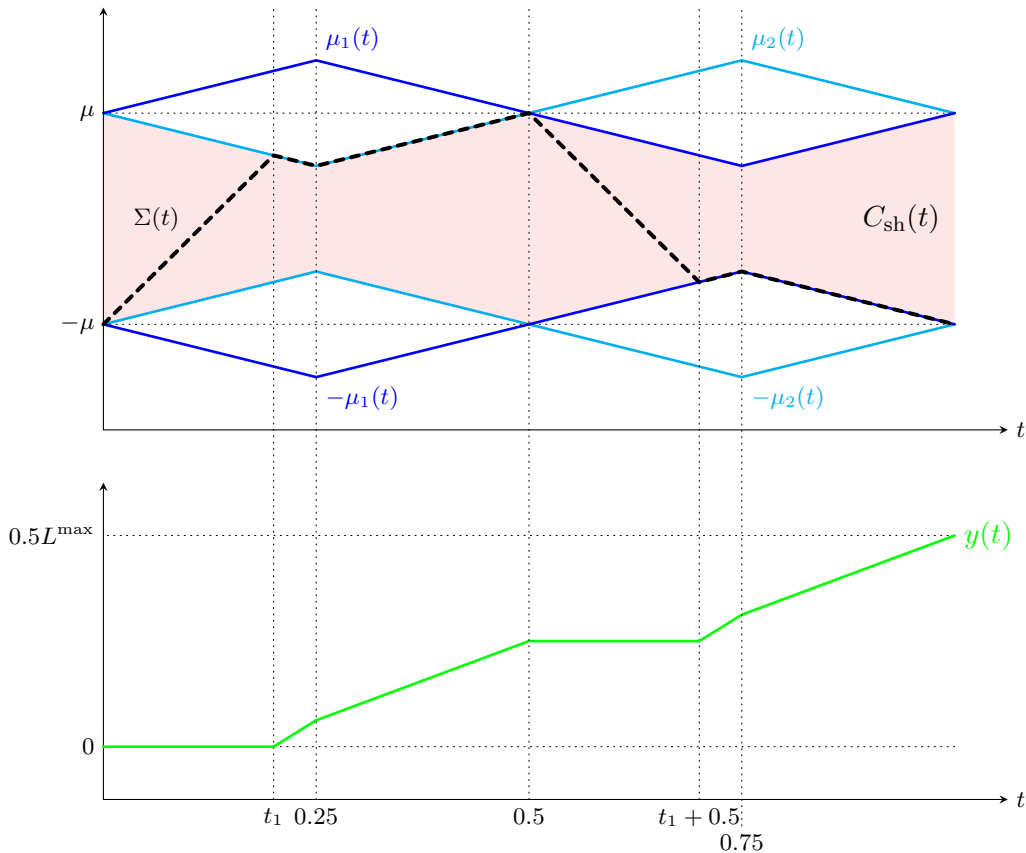
\begin{figure}[tb!]
	\begin{center}
		\begin{tikzpicture}[line cap=round,line join=round,>=stealth,x=0.7cm,y=0.7cm]
		\clip(-2.,-12.) rectangle (18.,5.); \footnotesize
		\fill[color=uuuuuu,fill=red,fill opacity=0.10] (0.,2.) -- (4.,1.) -- (8.,2.) -- (12.,1.) -- (16.,2.) -- (16.,-2.) -- (12.,-1.) -- (8.,-2.) -- (4.,-1.) -- (0.,-2.) -- cycle;
		\draw [line width=1.pt,color=cyan] (0.,2.)-- (4.,1.)-- (12.,3.) node[above right]{$\mu_2(t)$}-- (16.,2.);
		\draw [line width=1.pt,color=cyan] (0.,-2.)-- (4.,-1.)-- (12.,-3.)node[below right]{$-\mu_2(t)$}-- (16.,-2.);
		\draw [line width=1.pt,color=blue] (0.,2.)-- (4.,3.) node[above right]{$\mu_1(t)$} -- (12.,1.)-- (16.,2.);
		\draw [line width=1.pt,color=blue] (0.,-2.)-- (4.,-3.)node[below right]{$-\mu_1(t)$}-- (12.,-1.)-- (16.,-2.);
		\draw [line width=1.5pt,dashed] (0.,-2.) -- (3.2,1.2)-- (4.,1.)-- (8.,2.)-- (11.2,-1.2)-- (12.,-1.)-- (16.,-2.);
		\draw [->] (0.,-4.) -- (0.,4.);
		\draw [->] (0.,-4.) -- (17.,-4.) node[right]{$t$};
		\draw [dotted] (0.,2.) node[left]{$\mu$}-- (16.,2.);
		\draw [dotted] (0.,-2.) node[left]{$-\mu$}-- (16.,-2.);
		\draw [dotted] (16.,-10.)-- (0.,-10.) node[left]{$0$};
		\draw [dotted] (16.,-6.)-- (0.,-6.) node[left]{$0.5\Lmax$};
		\draw [dotted] (3.2,-11.) node[below]{$t_1$} -- (3.2,4.);
		\draw [dotted] (11.2,-11.) node[below]{$t_1 +0.5$}-- (11.2,4.);
		\draw [->] (0.,-11.) -- (0.,-5.);
		\draw [->] (0.,-11.) -- (17.,-11.)node[right]{$t$};;
		\draw [line width=1.pt,color=green] (0.,-10.)-- (3.2,-10.)-- (4.,-9.5)-- (8.,-8.)-- (11.2,-8)-- (12.,-7.5)-- (16.,-6.) node[right]{\normalsize $y(t)$};
		\draw [dotted] (4.,-11.) node[below]{$0.25$} -- (4.,4.);
		\draw [dotted] (8.,-11.) node[below]{$0.5$}-- (8.,4.);
		\draw [dotted] (12.,-11.5) node[below]{$0.75$}-- (12.,4.);
		\draw (1,0) node{$\ten(t)$};\normalsize
		\draw (15,0) node{$\Csh(t)$};
		\end{tikzpicture}
	\end{center}
\caption{Motility of the crawler of Example \ref{ex:stratA} (Strategy A). The picture above illustrates the evolution in time of the stasis domain $\Csh(t)$ (in red), produced by friction-manipulation. The dashed line described the evolution of the tension $\ten(t)$.\\	
The picture below shows the corresponding net translation $y(t)$ of the crawler. We notice that there are two advancing speeds. When $\mu_\mathrm{min}$ is decreasing (i.e.\ when $\Csh$ is shrinking) the crawler is faster, since a fraction of the internal energy of the crawler is released and spent in locomotion, in addition to the contribution given by actuation. On the other hand, when $\mu_\mathrm{min}$ is increasing (i.e.\ when $\Csh$ is expanding) a fraction of the actuation is  used to keep the system at the (increasing) critical tension, resulting in a smaller displacement.}
\label{fig:stratA}
\end{figure}

\begin{example}[Strategy B]\label{ex:stratB}
In this case the evolution of the friction coefficients is the same of the previous example, but with a different phase with respect to the actuation $L(t)$. More precisely, we consider
\begin{equation*}
\mu_1(t)=\begin{cases} (\frac{1}{2}+2t)\mu  &\text{for $0\leq t\leq \frac{1}{2}$}\\
(\frac{5}{2}-2t)\mu  &\text{for $\frac{1}{2}\leq t\leq 1$}
\end{cases}
\qquad
\mu_2(t)=\begin{cases} (\frac{3}{2}-2t)\mu  &\text{for $0\leq t\leq \frac{1}{2}$}\\
(2t-\frac{1}{2})\mu  &\text{for $\frac{1}{2} \leq t\leq 1$}
\end{cases}
\end{equation*}
As in the previous case, we start with $\ten(t)=-\mum(t)$.

We have three possible situations. 
If $k\Lmax<\mu$, then $\ten(t)$ will remain inside the stasis domain for the whole first half period, and no motion will occur. 

If $k\Lmax> 3\mu$, then $\ten(t)$ will reach the other boundary  of $\Csh$ at a time $t_3<\frac{1}{4}$ and $\ten(t)=\mum(t)$ will remain satisfied for all $t\in[t_3,\frac{1}{2}]$, with $z(t)$ increasing. More precisely
\begin{equation}
\dot z(t)=\begin{cases}
0 &\text{for $t\in(0,t_3)$}\\
2(\Lmax -\frac{\mu}{k}) &\text{for $t\in\left(t_3,\frac{1}{4}\right)$}\\
2(\Lmax +\frac{\mu}{k}) &\text{for $t\in\left(\frac{1}{4},\frac{1}{2}\right)$}
\end{cases} \qquad\text{with $t_3=\frac{\mu}{2(k\Lmax-\mu)}$}
\end{equation}

If $\mu\leq k\Lmax\leq 3\mu$, then $\ten(t)$ will reach the other boundary of $\Csh$ at a time $t_4>\frac{1}{4}$ and $\ten(t)=\mum(t)$ will remain satisfied for all $t\in[t_4,\frac{1}{2}]$, with $z(t)$ increasing. More precisely
\begin{equation}
\dot z(t)=\begin{cases}
0 &\text{for $t\in(0,t_4)$}\\
2(\Lmax +\frac{\mu}{k}) &\text{for $t\in\left(t_4,\frac{1}{2}\right)$}
\end{cases} \qquad\text{with $t_4=\frac{\mu}{k\Lmax+\mu}$}
\end{equation}

We notice that the increase of $z(t)$ produces a net advancement of the crawler if $\frac{1}{4}<t<\frac{1}{2}$, but a negative (backwards) net translation if $0<t<\frac{1}{4}$.
As in the previous example, the behaviour of the systems in the second half of the period can be recovered using the symmetries of the system, leading also in this case to \eqref{eq:inching_symm}.
Thus the net translation produced by a cycle is
$$
y(1)-y(0)=\begin{cases}
0& \text{if $ k\Lmax\leq \mu$} \\
\Lmax-\frac{\mu}{k}& \text{if $\mu\leq k\Lmax\leq 3\mu$} \\
2\frac{\mu}{k}& \text{if $k\Lmax\geq 3\mu$}
\end{cases}
$$
\end{example}

\begin{example}[Strategy C] \label{ex:stratC}
	The third strategy we propose resembles the one named \lq\lq crawling\rq\rq\ in \cite{Ume13}. At the state of maximum length, friction is maximum at both ends; at the state of minimum length, friction is minimum at both ends. During elongation (resp.~contraction) the friction increases (resp.~decreases) first in the back end and only then in the forward one. Namely we have
\begin{equation*}
	\mu_1(t)=\begin{cases} 
	(\frac{1}{2}+4t)\mu  &\text{for $0\leq t\leq \frac{1}{4}$}\\
	\frac{3}{2}\mu  &\text{for $\frac{1}{4}\leq t\leq \frac{1}{2}$}\\
	(\frac{7}{2}-4t)\mu  &\text{for $\frac{1}{2}\leq t\leq \frac{3}{4}$}\\
	\frac{1}{2}\mu  &\text{for $\frac{3}{4} \leq t\leq 1$}
	\end{cases}
	\qquad
	\mu_2(t)=\begin{cases} 
	\frac{1}{2}\mu  &\text{for $0\leq t\leq \frac{1}{4}$}\\
	(4t-\frac{1}{2})\mu  &\text{for $\frac{1}{4}\leq t\leq \frac{1}{2}$}\\
	\frac{3}{2}\mu  &\text{for $\frac{1}{2}\leq t\leq \frac{3}{4}$}\\
	(\frac{9}{2}-4t)\mu  &\text{for $\frac{3}{4} \leq t\leq 1$}
	\end{cases}
\end{equation*}
As in the previous two cases, we start with $\ten(t)=-\mum(t)$.
If $k\Lmax<2\mu$, the load is too small to cross the  stasis domain and no motion occurs.
We therefore consider $k\Lmax> 2\mu$. In this case $\ten(t)$ will reach the other boundary  of $\Csh$ at a time $t_5<\frac{1}{4}$ and $\ten(t)=\mum(t)$ will remain satisfied for all $t\in[t_5,\frac{1}{2}]$, with $z(t)$ increasing. More precisely
\begin{equation}
\dot z(t)=\begin{cases}
0 &\text{for $t\in(0,t_5)$}\\
2\Lmax &\text{for $t\in\left(t_5,\frac{1}{4}\right)$}\\
2(\Lmax -2\frac{\mu}{k}) &\text{for $t\in\left(\frac{1}{4},\frac{1}{2}\right)$}
\end{cases} \qquad\text{with $t_5=\frac{\mu}{2k\Lmax}$}
\end{equation}
Then, after half period we  have $\ten(1/2)=\frac{3}{2}\mu$. 

The behaviour during the second half period is not symmetric to that observed in the first half period; indeed, we can distinguish two different qualitative scenarios.
If $2\mu<k\Lmax<4\mu$, then $\ten(t)$ reaches the boundary of $\Csh$ at a time $t_6>\frac{3}{4}$, when $-\mum(t)$ is already constant, and $\ten(t)=-\mum(t)$ remains satisfied until $t=1$. More precisely we have 
\begin{equation}
\dot z(t)=\begin{cases}
0 &\text{for $t\in(\frac{1}{2},t_6)$}\\
-2\Lmax &\text{for $t\in\left(t_6,1\right)$}
\end{cases} \qquad\text{with $t_6=\frac{1}{2}+\frac{\mu}{k\Lmax}$}
\end{equation}
If instead $k\Lmax> 4\mu$, then $\ten(t)$ will reach the lower boundary  of $\Csh$ at a time $t_7<\frac{3}{4}$ and $\ten(t)=\mum(t)$ will remain satisfied for all $t\in[t_7,1]$, with $z(t)$ decreasing. More precisely
\begin{equation}
\dot z(t)=\begin{cases}
0 &\text{for $t\in(\frac{1}{2},t_7)$}\\
-2(\Lmax +2\frac{\mu}{k}) &\text{for $t\in\left(t_7,\frac{3}{4}\right)$}\\
-2\Lmax &\text{for $t\in\left(\frac{3}{4},1\right)$}
\end{cases} \qquad\text{with $t_7=\frac{1}{2}+\frac{3\mu}{2k\Lmax+4\mu}$}
\end{equation}

Using these results, we can compute the net translation produced by a cycle as
$$
y(1)-y(0)=\begin{cases}
0& \text{if $ k\Lmax\leq 2\mu$} \\
\Lmax -2\frac{\mu}{k}& \text{if $k\Lmax\geq 2\mu$}
\end{cases}
$$
\end{example}

\medskip
We observe that strategies A and C implement, in different ways, the main idea of two-anchor crawling, using the changes in the friction coefficient to move the two contact points only in the desired direction. As a consequence, the net translation gained in each period is equal to the amplitude $\Lmax$ of the shape actuation, minus a constant term due to elasticity. Indeed, the net translation after a period is the same that we would obtain in the case of constant anisotropic friction, with the lower friction coefficient equal to $\mu$  (cf.~Example~\ref{ex:2legs}). 

We remark, however, that strategy A has two advantages with respect to strategy C. First, strategy A can be described by a  single control parameter, since $\mu_2(t)=2\mu-\mu_1(t)$, whereas strategy C can be obtained only using two control parameters. Second, strategy C requires a faster change in the friction coefficients.

The comparison between strategies A and B emphasizes the importance of the phase between shape actuation and friction manipulation. While a phase shift of strategy A of half period would produce a backwards translation, symmetric to Example \ref{ex:stratA}, strategy B illustrates a possible behaviour at intermediate shifts.  
Interestingly, strategy B has a lower minimal amplitude in shape actuation necessary to achieve locomotion, hence for small value of $\Lmax$ it is actually more efficient than strategy A. However, this holds until a certain threshold: large values of $\Lmax$ produce a temporary backward displacement of each contact points, and the net translation after a period does not increase any longer with $\Lmax$.

\paragraph{Acknowledgements}
The author is supported by FCT -- Fundação para a Ciência e Tecnologia, under the project UID/MAT/04561/2013.

\printbibliography

\end{document}